\documentclass{article}
\usepackage{tikz-cd}
\usepackage{fullpage}
\usepackage{authblk}
\usepackage{amsmath}
\usepackage{amsthm}
\usepackage{amssymb}
\usepackage{bm}
\usepackage{setspace}
\usepackage[numbers,sort&compress]{natbib}
\usepackage{color}

\usepackage{amsfonts}
\usepackage{eucal}
\usepackage{latexsym}
\usepackage{mathdots}
\usepackage{mathrsfs}
\usepackage{enumitem}

\newcommand{\be}{\begin{equation}}
\newcommand{\ee}{\end{equation}}
\newcommand{\ba}{\begin{eqnarray*}}
\newcommand{\ea}{\end{eqnarray*}}
\newcommand{\bal}{\begin{align}}
\newcommand{\eal}{\end{align}}
\newcommand{\baln}{\begin{align*}}
\newcommand{\ealn}{\end{align*}}
\newcommand{\bi}{\begin{itemize}}
\newcommand{\ei}{\end{itemize}}
\newcommand{\bn}{\begin{enumerate}}
\newcommand{\en}{\end{enumerate}}
\newcommand{\bbm}{\begin{bmatrix}}
\newcommand{\ebm}{\end{bmatrix}}
\newcommand{\bpm}{\begin{pmatrix}}
\newcommand{\epm}{\end{pmatrix}}
\newcommand{\bsm}{\left ( \begin{smallmatrix}}
\newcommand{\esm}{\end{smallmatrix} \right) }
\newcommand{\bp}{\begin{proof}}
\newcommand{\ep}{\end{proof}}

\newcommand{\mr}{\ensuremath{\mathrm}}
\newcommand{\scr}{\ensuremath{\mathscr}}

\newcommand{\mf}{\ensuremath{\mathfrak}}
\newcommand{\ov}{\ensuremath{\overline}}

\newcommand{\wt}{\ensuremath{\widetilde}}

\newcommand{\ga}{\ensuremath{\gamma}}
\newcommand{\Om}{\ensuremath{\Omega}}

\newcommand{\la}{\ensuremath{\lambda }}

\newcommand{\eps}{\ensuremath{\epsilon }}

\def\C{\mathbb{C}}
\def\R{\mathbb{R}}
\def\D{\mathbb{D}}

\def\N{\mathbb{N}}

\def\fq{\mathfrak{q}}

\def\nbran{\mathrm{Ran} \, }
\def\nbdom{\mathrm{Dom} \, }
\def\nbker{\mathrm{Ker} \, }
\def\nbre{\mathrm{Re} \, }
\def\nbim{\mathrm{Im} \, }
\newcommand{\cH}{\ensuremath{\mathcal{H}}}
\newcommand{\cJ}{\ensuremath{\mathcal{J} }}

\newcommand{\Addresses}{{% additional braces for segregating \footnotesize
  \bigskip

Jashan~Bal, \textsc{Department of Pure Mathematics, University of Waterloo}\par\nopagebreak
\textit{E-mail address:} \texttt{j2bal@uwaterloo.ca}
    
\medskip

Robert~T.~W.~Martin, \textsc{Department of Mathematics, University of Manitoba}\par\nopagebreak
  \textit{E-mail address:} \texttt{Robert.Martin@umanitoba.ca}

\medskip

Fouad~Naderi, \textsc{Department of Mathematics, University of Manitoba}\par\nopagebreak
  \textit{E-mail address:} \texttt{naderif@myumanitoba.ca}

\vspace{1cm}

}}

\newcommand{\ip}[2]{\ensuremath{\left\langle {#1} , {#2} \right\rangle}}

\newtheorem{thm}{Theorem}

\newtheorem{lemma}{Lemma}
\newtheorem{prop}{Proposition}
\newtheorem{cor}{Corollary}
\newtheorem*{thm*}{Theorem}

\theoremstyle{definition}
\newtheorem{defn}{Definition}
\newtheorem{remark}{Remark}
\newtheorem{eg}{Example}

\title{A reproducing kernel approach to Lebesgue decomposition}

\author[1]{Jashan Bal}
\affil[1]{\footnotesize University of Waterloo}

\author[2]{Robert T.W. Martin\thanks{Supported by NSERC grant 2020-05683}}

\author[2]{Fouad Naderi}
\affil[2]{\footnotesize University of Manitoba}

\date{\vspace{-1.2cm}}

\begin{document}
\maketitle
\begin{abstract}
We show that properties of pairs of finite, positive and regular Borel measures on the complex unit circle such as domination, absolute continuity and singularity can be completely described in terms of containment and intersection of their reproducing kernel Hilbert spaces of `Cauchy transforms' in the complex unit disk. This leads to a new construction of the classical Lebesgue decomposition and proof of the Radon--Nikodym theorem using reproducing kernel theory and functional analysis. \\
\end{abstract}

Given a finite, positive and regular Borel measure, $\mu$, on the complex unit circle, $\partial \D$, it is natural to consider its $L^2-$space, $L^2 (\mu ) := L^2 (\mu , \partial \D )$, as well as $H^2 (\mu ) := \C [ \zeta ] ^{-\| \cdot \| _{L^2(\mu )}}$, the closure of the analytic polynomials, $\C [\zeta ]$, in $L^2 (\mu )$. The linear operator of multiplication by the independent variable, $M^\mu _\zeta$, is unitary on $L^2 (\mu )$ and has $H^2 (\mu )$ as a closed invariant subspace so that $Z^\mu := M^\mu _\zeta | _{H^2 (\mu )}$ is an isometry that will play a central role in our analysis. The $\mu-$Cauchy transform of any $h \in H^2 (\mu )$ is the analytic function,
$$ (\scr{C} _\mu h) (z) := \int _{\partial \D} \frac{h (\zeta)}{1- z \ov{\zeta}} \mu (d\zeta ) \in \scr{O} (\D ), $$ in the complex unit disk, $\D := (\C ) _1$. Here, given a Banach space, $X$, $(X) _1$ and $[X]_1$ denote its open and closed unit balls in the norm topology. 

Recall that a reproducing kernel Hilbert space (RKHS), $\cH$, is a Hilbert space of functions on a set, $X$, so that point evaluation at any point $x \in X$ yields a bounded linear functional on the space. The Riesz representation lemma then implies the existence of \emph{kernel vectors}, $k_x$, $x \in X$, so that the bounded linear functional of point evaluation at $x$ is implemented by inner products against $k_x$. The function of two variables, $k : X \times X \rightarrow \C$, 
$$ k(x,y) := \ip{k_x}{k_y}_{\cH}, $$ is then called the \emph{reproducing kernel} of $\cH$. In this paper, all inner products and sesquilinear forms are conjugate linear in their first argument and linear in their second argument. Any reproducing kernel function is a \emph{positive kernel} function on $X \times X$, \emph{i.e.} for any finite set $\{ x_1, \cdots, x_n \} \subseteq X$, the $n \times n$ matrix,
\be [ k (x_i , x_j) ] _{1 \leq i,j \leq n} \geq 0, \label{poskernel} \ee is positive semi-definite. Conversely, by a theorem of Aronszajn and Moore, given any positive kernel function, $k$, on $X \times X$, one can construct a RKHS of functions on $X$ with reproducing kernel $k$ \cite{Aron-rkhs}, see Subsection \ref{ss-rkhs}. Given this bijective correspondence between positive kernel functions on $X$ and RKHS of functions on $X$, one writes $\cH = \cH (k)$ if $\cH$ is a RKHS with reproducing kernel $k$. 

Equipping the vector space of $\mu-$Cauchy transforms with the $H^2 (\mu)-$inner product, 
$$ \ip{\scr{C} _\mu g}{\scr{C} _\mu h}_{\mu} := \ip{g}{h}_{L^2 (\mu )}; \quad \quad g,h \in H^2 (\mu ), $$ yields a reproducing kernel Hilbert space (RKHS) of analytic functions in $\D$, $\scr{H} ^+ ( \mu )$, with reproducing kernel, 
\be k^\mu (z,w ) = \int _{\partial \D} \frac{1}{1-z \ov{\zeta}} \frac{1}{1-\ov{w}\zeta} \mu (d\zeta ); \quad \quad z,w \in \D. \label{repkernel} \ee  
Using the above formula (\ref{repkernel}), it is easy to check that domination of measures implies domination of the reproducing kernels for their spaces of Cauchy transforms,
$$ \mu \leq t^2 \la, \ t >0 \quad \quad \Rightarrow \quad \quad k^\mu \leq t^2 k^\la, $$ see Theorem \ref{RKdominate}, where we write $k \leq K$ for positive kernel functions $k,K$ on $X$, if $K-k$ is a positive kernel function on $X$. We will say that \emph{$\la$ dominates $\mu$ in the reproducing kernel sense} (by $t^2 >0$) and write $\mu \leq _{RK} t^2 \la$ to denote that $k^\mu \leq t^2 k^\la$. By results of Aronszajn, domination of kernels, $k \leq t^2 K$, is equivalent to bounded containment of their RKHS, \emph{i.e.} $k \leq t^2 K$ if and only if $\cH (k) \subseteq \cH (K)$ and the norm of the linear embedding $\mr{e} : \cH (k) \hookrightarrow \cH (K)$ is at most $t >0$ \cite{Aron-rkhs}. See Subsection \ref{ss-rkhs} for a review of RKHS theory and these results. In summary, domination of measures implies bounded containment of their spaces of Cauchy transforms:
$$ \mu \leq t^2 \la \quad \Rightarrow \quad \scr{H} ^+ (\mu ) \subseteq \scr{H} ^+ (\la ),  \ \mr{e} _{\mu ,\la} : \scr{H} ^+ (\mu ) \hookrightarrow \scr{H} ^+ (\la), \ \| \mr{e} _{\mu, \la} \| \leq t, $$ \emph{i.e.} $\mu \leq t^2 \la$ $\Rightarrow$ $\mu \leq _{RK} t^2 \la$.

Building on this observation, we show that domination and, more generally, absolute continuity, as well as mutual singularity of measures can be completely characterized in terms of their spaces of Cauchy transforms. Moreover, we develop an independent construction of the Lebesgue decomposition and new proof of the Radon--Nikodym theorem using reproducing kernel methods and operator theory. 

\subsection*{Outline}

The following Background section, Section \ref{s-back}, provides an introduction to (i) the bijective correspondence between positive, finite and regular Borel measures on the circle and contractive analytic functions in the disk, (ii) reproducing kernel theory and (iii) the theory of densely--defined and positive semi-definite quadratic forms in a separable, complex Hilbert space.

Section \ref{CTsec} introduces the reproducing kernel Hilbert spaces, $\scr{H} ^+ (\mu )$, of $\mu-$Cauchy transforms associated to any positive, finite and regular Borel measure, $\mu$, on the complex unit circle. These are Hilbert spaces of holomorphic functions in the complex unit disk. 

Our first main results appear in Section \ref{sec:RKac}. Theorem \ref{RKdominate} proves that domination of positive measures in the reproducing kernel sense is equivalent to domination in the classical sense:

\begin{thm*}[Theorem \ref{RKdominate}]
Given positive, finite and regular Borel measures, $\mu$ and $\la$ on the unit circle, $\mu \leq _{RK} t^2 \la$ for some $t>0$ if and only if $\mu \leq t^2 \la$.
\end{thm*}

This result is extended to general absolute continuity, written $\mu \ll \la$, in Theorem \ref{RKac}. Namely, we say that $\mu$ is absolutely continuous in the reproducing kernel sense with respect to $\la$, written $\mu \ll _{RK} \la$, if the intersection of the space of $\mu-$Cauchy transforms with the space of $\la-$Cauchy transforms, $\mr{int} (\mu, \la)$, is norm-dense in $\scr{H} ^+ (\mu )$.

\begin{thm*}[Theorem \ref{RKac}]
Let $\mu, \la$ be positive, finite and regular Borel measures on $\partial \D$. Then $\mu \ll \lambda$ if and only if $\mu \ll _{RK} \la$.
\end{thm*}

Moreover, Theorem \ref{RKac} gives a formula for the Radon--Nikodym derivative of $\mu$ with respect to $\la$ in terms of the closed, densely--defined embedding, $\mr{e} _{\mu, \la} : \mr{int} (\mu , \la) := \scr{H} ^+ (\mu ) \cap \scr{H} ^+ (\la ) \subseteq \scr{H} ^+ (\mu ) \hookrightarrow \scr{H} ^+ (\la )$. 

These are satisfying results, however, actual construction of the Lebesgue decomposition of $\mu$ with respect to $\la$ using reproducing kernel methods is more subtle and bifurcates into the two cases, where: The intersection space, $\mr{int} (\mu , \la) = \scr{H} ^+ (\mu ) \cap \scr{H} ^+ (\la )$, of the spaces of $\mu$ and $\la-$Cauchy transforms is (i) invariant, or, (ii) not invariant, for the image, $V^\mu$, of $Z^\mu = M^\mu _\zeta | _{H^2 (\mu)}$ under Cauchy transform. Some necessary and sufficient conditions for this to hold are obtained in Lemma \ref{invint} and Proposition \ref{notred}. Namely, as described in Subsection \ref{ss:Clark}, there is a bijection between contractive analytic functions in the complex unit disk and positive, finite and regular Borel measures on the circle. If a positive measure, $\mu$, corresponds to an extreme point of this compact, convex set of contractive analytic functions, we say that $\mu$ is \emph{extreme}, otherwise $\mu$ is \emph{non-extreme}. As established in Lemma \ref{invint} and Proposition \ref{notred}, the intersection space, $\mr{int} (\mu , \la)$ will be $V_\mu-$reducing if (i) $\la$ is non-extreme or if (ii) $\mu +\la$ is extreme, and the intersection space will be non-trivial and not $V_\mu-$invariant if $\mu, \la$ are both extreme but $\mu + \la$ is non-extreme.

In the positive direction, we obtain:

\begin{thm*}[Theorem \ref{LDviaRKne}]
Let $\mu$ and $\la$ be finite, positive and regular Borel measures on the unit circle. If the intersection space, $\mr{int} (\mu , \la)$ is $V_\mu-$invariant and $\mu = \mu _{ac} + \mu_s$ is the Lebesgue decomposition of $\mu$ with respect to $\la$, then 
$$ \scr{H} ^+ (\mu ) = \scr{H} ^+ (\mu _{ac} ) \oplus \scr{H} ^+ (\mu _s). $$ In this case,
$$ \scr{H} ^+ (\mu _{ac} ) = \mr{int} (\mu , \la) ^{-\| \cdot \| _\mu }, \quad \mbox{and} \quad \scr{H} ^+ (\mu _s ) \cap \scr{H} ^+ (\la ) = \{ 0 \}. $$
\end{thm*}

Given two positive, finite and regular Borel measures, $\mu$ and $\la$, on the complex unit circle, $\partial \D$, one can associate to $\mu$ a densely--defined and positive semi-definite sesquilinear or quadratic form in $H^2 (\la )$.  Namely, we define the \emph{form domain}, $\nbdom \fq_\mu \subseteq H^2 (\la)$, as the \emph{disk algebra}, $\nbdom \fq _\mu := A(\D )$, the unital Banach algebra of all uniformly bounded analytic functions in the unit disk which extend continuously to the boundary, equipped with the supremum norm. The disk algebra embeds isometrically into the continuous functions on the circle, $\scr{C} (\partial \D )$ and $A (
\D )$ can be viewed as a dense subspace of $H^2 (\la)$. The quadratic form, $\fq_\mu : \nbdom \fq_\mu \times \nbdom \fq_\mu \rightarrow \C$ is then defined in the obvious way by integration against $\mu$,
\be \fq_\mu (g, h) := \int _{\partial \D} \ov{g(\zeta )} h(\zeta) \mu (d\zeta ), \quad \quad g,h \in A (\D ) = \nbdom \fq _\mu. \ee As described in Section \ref{sec:RKac} and Theorem \ref{formLD}, there is a theory of Lebesgue decomposition of densely--defined and positive semi-definite quadratic forms in a Hilbert space, $\cH$. Namely, given any such form, there is a unique \emph{Simon--Lebesgue form decomposition}, 
$$ \fq= \fq_{ac} + \fq_s, $$ where $0 \leq \fq_{ac}, \fq_s \leq \fq$, $\fq_{ac}$ is \emph{absolutely continuous} in the sense that it is \emph{closeable} and it is maximal in the sense that $\fq_{ac}$ is the largest closeable quadratic form bounded above by $\fq$. The form $\fq_s$ is singular in the sense that the only closeable positive semi-definite form it dominates is the identically $0$ form. Here, a positive semi-definite quadratic form, $\fq$, with dense form domain $\nbdom \fq$ in $\cH$, is \emph{closed}, if $\nbdom \fq$ is a Hilbert space, \emph{i.e.} complete, with respect to the norm induced by the inner product $\fq ( \cdot , \cdot ) + \ip{\cdot}{\cdot}_{\cH}$. A form is then \emph{closeable} if it has a closed extension. See Subsection \ref{ss:forms} for an introduction to the theory of densely--defined and positive semi-definite quadratic forms.

An immediate question is whether the Simon--Lebesgue decomposition of the form, $\fq_\mu$, in $H^2 (\la )$ coincides with the Lebesgue decomposition of $\mu$ with respect to $\la$. Namely, if $\mu = \mu _{ac} + \mu _s$ and $\fq_\mu = \fq_{ac} + \fq_s$, then is it true that $\fq_{ac} = \fq_{\mu _{ac}}$ and $\fq_s = \fq_{\mu _s}$? A complete answer, summarized in the theorem below, is provided in Theorem \ref{rkhsform}, Theorem \ref{genRKLD}, Corollary \ref{RKformLD} and Corollary \ref{LDvsformLD}.

\begin{thm*}
If $\fq_{\mu} = \fq_{ac} + \fq_s$ is the Simon--Lebesgue form decomposition of $\fq_\mu$ in $H^2 (\la)$, then 
$$ \scr{H} ^+ (\mu ) = \scr{H} ^+ (\fq_{ac} ) \oplus \scr{H} ^+ (\fq _s), $$ where 
$$ \scr{H} ^+ (\fq_{ac} ) = \mr{int} (\mu , \la ) ^{-\| \cdot \|_\mu}. $$ If $\mu = \mu _{ac} + \mu _s$ is the Lebesgue decomposition of $\mu$ with respect to $\la$, then
$$ \scr{H} ^+ (\mu ) = \scr{H} ^+ (\mu _{ac} ) + \scr{H} ^+ (\mu _s ), $$ is a complementary space decomposition in the sense of de Branges and Rovnyak, with $\scr{H} ^+ (\mu _{ac} )$, $\scr{H} ^+ (\mu _s)$ contractively contained in $\scr{H} ^+ (\mu )$. Moreover, $\scr{H} ^+ (\mu _{ac} )$ is the largest RKHS, $\cH (k)$, contractively contained in $\scr{H} ^+ (\fq _{ac} ) \subseteq \scr{H} ^+ (\mu )$ so that the closed embedding, $\mr{e} : \cH (k) \cap \scr{H} ^+ (\la ) \subseteq \cH (k) \hookrightarrow \scr{H} ^+ (\la )$, is such that $\tau := \mr{e} \mr{e} ^*$ is Toeplitz for the image, $V^\la$, of $Z^\la$ under Cauchy transform, \emph{i.e.} $V ^{\la *} \tau V ^\la = \tau$.
In particular, the Simon--Lebesgue decomposition of the quadratic form, $\fq_\mu$, in $H^2 (\la )$ coincides with the Lebesgue decomposition of $\mu$ with respect to $\la$ if and only if $\mr{int} (\mu , \la)$ is $V^\mu-$invariant.
\end{thm*}
In the above, the spaces of $\fq _{ac}$ and $\fq _s-$Cauchy transforms are defined in an analogous way to the space of $\mu-$Cauchy transforms, see Subsection \ref{ss:measforms}. By Proposition \ref{notred}, the intersection space, $\mr{int} (\mu , \la )$, is not always $V^\mu-$invariant. Example \ref{LebesgueEg} (continued in Example \ref{Lebeg2}) provides a concrete example, where $\mu = m_+$ and $\la = m_-$ are the mutually singular restrictions of normalized Lebesgue measure, $m$, to the upper and lower half-circles, so that the Lebesgue decomposition of $m_+$ with respect to $m_-$ has $m_{+; ac} =0$ but $\mr{int} (m_+ ,m_-) \neq \{ 0 \}$, so that $\fq _{m_+ ; ac} \neq 0$. 

\begin{remark}
 This `reproducing kernel approach' to measure theory on the circle and Lebesgue decomposition of a positive measure with respect to Lebesgue measure was first considered and studied in \cite{JM-ncFatou,JM-ncld}, in a more general and non-commutative context.
\end{remark}

\section{Background} \label{s-back}

\subsection{Function theory in the disk, measure theory on the circle}
\label{ss:Clark}

Classical analytic function theory in the complex unit disk and measure theory on the complex unit circle are fundamentally intertwined. There are bijective correspondences between (i) contractive analytic functions in the disk, (ii) analytic functions in the disk with positive semi-definite real part, \emph{i.e.} \emph{Herglotz functions} and (iii) positive, finite and regular Borel measures on the complex unit circle. Namely, starting with such a positive measure, $\mu$, its \emph{Herglotz--Riesz transform} is the Herglotz function,
$$ H _\mu (z) := \int _{\partial \D} \frac{1+z \ov{\zeta}}{1-z \ov{\zeta}} \mu (d\zeta ) \in \scr{O} (\D ). $$ It is easily verified that $\nbre H_\mu (z) \geq 0$, is a positive harmonic function.  
Applying the inverse \emph{Cayley transform} to any Herglotz function, \emph{i.e.} the M\"obius transformation sending the open right half-plane onto the open unit disk, $\D$, which interchanges the points $1$ and $0$, yields a contractive analytic function, $b_\mu$, in the disk,
$$ b_\mu (z) := \frac{H_\mu (z) -1}{H_\mu (z) +1}, \quad \quad |b _\mu (z) | \leq 1, \ z \in \D. $$ (By the maximum modulus principle, $b_\mu$ is strictly contractive in $\D$ unless it is constant.) Each of these transformations is essentially reversible. Namely, given any contractive analytic function, $b$, the Cayley transform, $H_b := \frac{1 + b}{1-b}$, is a Herglotz function and the Herglotz representation theorem states that if $H$ is any Herglotz function in the disk, then there is a unique finite, positive and regular Borel measure, $\mu$ on the circle, so that 
$$ H (z) = i \nbim H (0) + \int _{\partial \D} \frac{1+z \ov{\zeta}}{1-z \ov{\zeta}} \mu (d\zeta ) = i \nbim H (0) + H_\mu (z), $$ see \cite[Boundary Values, Chapter 3]{Hoff}. To be precise, two Herglotz functions correspond to the same positive measure, $\mu$, if and only if they differ by an imaginary constant. If $H_1, H_2$ are two Herglotz functions so that $H_2 = H_1 + it$ for some $t \in \R$, then their corresponding inverse Cayley transforms obey
$$b _2 = \frac{\ov{z(t)}}{z(t)} \cdot \mf{m} _{z(t)} \circ b_1, \quad \quad \mf{m} _{z(t)} (z) = \frac{z - z(t)}{1-\ov{z(t)} z}, \ z(t) := \frac{t}{2i+t} \in \D,$$ so that $b_2$ is, up to multiplication by the unimodular constant $\frac{\ov{z(t)}}{z(t)}$, a M\"obius transformation, $\mf{m} _{z(t)}$, of $b_1$, where $\mf{m} _{z(t)}$ defines an automorphism of the disk interchanging $0$ with $z(t)$.

If a contractive analytic function, $b$, corresponds, essentially uniquely, to a positive measure, $\mu$, in this way, we write $\mu := \mu _b$, and $\mu _b$ is called the \emph{Clark} or \emph{Aleksandrov--Clark measure} of $b$ \cite{Clark}. Many properties of contractive analytic functions in the disk can be described in terms of corresponding properties of their Clark measures and vice versa \cite{Aleks1,Aleks2}. For example, by Fatou's theorem, the Radon--Nikodym derivative of any Clark measure, $\mu _b$, with respect to normalized Lebesgue measure, $m$, on the circle is given by the radial, or more generally non-tangential, limits of the real part of its Herglotz function,
\ba \frac{\mu _b (d\zeta )}{m(d\zeta )} & = & \lim _{r \uparrow 1} \nbre H_b (r\zeta ); \quad \quad m-a.e., \ \zeta \in \partial \D  \\
& = & \lim _{r \uparrow 1} \frac{ 1 - | b (r \zeta ) | ^2 }{| 1 - b (r \zeta )| ^2} \geq 0, \ea
\cite{Fatou}, \cite[Fatou's Theorem, Chapter 3]{Hoff}.
As a corollary of this formula, we see that $b$ is \emph{inner}, \emph{i.e.} it has unimodular radial boundary limits $m-$a.e. on the circle, if and only if its Radon--Nikodym derivative vanishes almost everywhere, \emph{i.e.} if and and only if its Clark measure is singular with respect to Lebesgue measure. 

As a second example which will be relevant for our investigations here, $b$ is an extreme point of the closed convex set of contractive analytic functions in the disk if and only if its Radon--Nikodym derivative with respect to Lebesgue measure is not log-integrable.  That is, $b$ is an extreme point if and only if 
$$ \mr{log} \, \frac{\mu _b (d\zeta )}{m (d\zeta)} \ \notin \ L^1 = L^1 (m). $$ This follows from the characterization of extreme points in the set of contractive analytic functions given in \cite[Extreme Points, Chapter 9]{Hoff} and Fatou's Radon--Nikodym formula as described above. Here, equipping the set of all bounded analytic functions in the disk with the supremum norm, we obtain the unital Banach algebra, $H^\infty$, the Hardy algebra, whose closed unit ball, $[H^\infty ] _1$, is the compact and convex set of contractive analytic functions in the disk. It further follows from a well-known theorem of Szeg\"o (later strengthened by Kolmogoroff and Kre\u{\i}n), that $H^2 (\mu ) = L^2 (\mu )$ if and only if $\mu = \mu _b$ for an extreme point $b \in [ H^\infty ] _1$ \cite[Szeg\"o's Theorem, Chapter 4]{Hoff}, \cite{Szego-thm}. Namely, Szeg\"o's theorem gives a formula for the distance from the constant function $1$ to the closure of the analytic polynomials with zero constant term in $L^2 (\mu )$:
$$ \mr{inf}_{\substack{p \in \C  [\zeta ] \\ p(0) = 0}} \| 1 -p \| ^2 _{L^2 (\mu )} = \mr{exp} \, \int _{\partial \D} \mr{log} \, \frac{\mu (d\zeta )}{m (d\zeta)} \, m(d\zeta). $$ It follows, in particular, that $b$ is an extreme point so that $\frac{d\mu}{dm}$ is not log-integrable if and only if $1$ belongs to the closure, $H^2 _0 (\mu )$, in $L^2 (\mu)$ of the analytic polynomials obeying $p(0) =0$. That is, if and only if $H^2 _0 (\mu ) = H^2 (\mu )$. An inductive argument then shows that this is equivalent to $H^2 (\mu ) = L^2 (\mu )$, so that $Z^\mu = M^\mu _\zeta | _{H^2 (\mu)} = M^\mu _\zeta$ is unitary. If $\mu = \mu _b$ is the Clark measure of an extreme point, $b$, we will say that $\mu$ is \emph{extreme}, and that $\mu$ is \emph{non-extreme} if $b$ is not an extreme point.

The results of this paper reinforce the close relationship between function theory in the disk and measure theory on the circle by establishing the Lebesgue decomposition and Radon--Nikodym theorem for positive measures using functional analysis and reproducing kernel theory applied to spaces of Cauchy transforms of positive measures. We will see that the reproducing kernel construction of the Lebesgue decomposition of a positive measure $\mu$, with respect to another, $\la$, bifurcates into the two cases, where: the intersection of the spaces of $\mu$ and $\la-$Cauchy transforms, is (i) invariant, or (ii), not invariant for the image of $Z ^\mu$ under Cauchy transform. Moreover, whether or not this intersection space is invariant is largely dependent on whether $\la$, or $\mu + \la$ are non-extreme or extreme. 

\subsection{Reproducing kernel Hilbert spaces}

\label{ss-rkhs}

As described in the introduction, a reproducing kernel Hilbert space (RKHS) is any complex, separable Hilbert space of functions, $\cH$, on a set $X$, with the property that the linear functional of point evaluation at any $x \in X$ is bounded on $\cH$. Further recall, as described above, that for any $x \in X$, there is then a unique \emph{kernel vector} or \emph{point evaluation vector}, $k_x \in \cH$ so that $\ip{k_x}{h}_\cH = h(x)$ for any $h \in \cH$ and we write $\cH = \cH (k)$, where $k : X \times X \rightarrow \C$ is a \emph{positive kernel function} on $X$ in the sense of Equation (\ref{poskernel}). Much of elementary reproducing kernel Hilbert space theory was developed by N. Aronszajn in his seminal paper, \cite{Aron-rkhs}. In particular, there is a bijective correspondence between RKHS on a set $X$ and positive kernel functions on $X$ given by the Aronszajn--Moore theorem, \cite[Part I]{Aron-rkhs}, \cite[Proposition 2.13, Theorem 2.14]{Paulsen-rkhs} and this motivates the notation $\cH = \cH (k)$. 

\begin{thm*}[Aronszajn--Moore]
If $\cH = \cH (k)$ is a RKHS of functions on a set, $X$, then $k$ is a positive kernel function on $X$. Conversely, if $k$ is a positive kernel function on $X$, then there is a (necessarily unique) RKHS of functions on $X$ with reproducing kernel, $k$.
\end{thm*}

Any RKHS, $\cH (k)$, of functions on a set $X$, is naturally equipped with a \emph{multiplier algebra}, $\mr{Mult} (k)$, the unital algebra of all functions on $X$ which `multiply' $\cH (k)$ into itself. That is, $g \in \mr{Mult} (k)$ if and only if $g\cdot h \in \cH (k)$ for any $h \in \cH (k)$. Any $h \in \mr{Mult} (k)$ can be identified with the linear multiplication operator $M_h : \cH (k) \rightarrow \cH (k)$. More generally, one can consider the set of multipliers, $\mr{Mult} (k, K)$, between two RKHS on the same set. If $h \in \mr{Mult} (k, K)$, then $M_h$ is always bounded, by the closed graph theorem. Adjoints of multiplication operators have a natural action on kernel vectors: If $h \in \mr{Mult} (k , K)$, then 
$$ M_h ^* K_z = k_z \ov{h(z)}. $$ 

All RKHS in this paper will be RKHS, $\cH (k)$, of analytic functions in the complex unit disk, $\D = (\C ) _1$, with the additional property that evaluation of the Taylor coefficients of any $h \in \cH (k)$ (at $0$) defines a bounded linear functional on $\cH (k)$. Again, by the Riesz representation lemma, for any $j \in \N \cup \{ 0 \}$, there is then a unique \emph{Taylor coefficient kernel vector}, $k_j \in \cH (k)$, so that if $h \in \cH (k)$ has Taylor series at $0$,
$$ h(z) = \sum _{j=0} ^\infty \hat{h} _j z^j, $$ then $\ip{k_j}{h}_{\cH (k)} = \hat{h} _j$. It follows that 
$$ \hat{k} (i , j) := \ip{k_i}{k_j}_{\cH (k)}, $$ defines a positive kernel function, the \emph{coefficient reproducing kernel} of $\cH (k)$, on the set $\N \cup \{ 0 \}$. It is easily checked that for any such Taylor coefficient reproducing kernel Hilbert spaces, $\cH (k)$ and $\cH (K)$, of analytic functions in $\D$,
$$ k \leq K \quad \Leftrightarrow \quad \hat{k} \leq \hat{K}. $$
The reproducing and coefficient reproducing kernels of a Taylor coefficient RKHS in $\D$ are related by the formulas:
$$ k(z,w) = \sum _{j,\ell =0} ^\infty \hat{k} (j , \ell) z^j \ov{w} ^\ell, \quad \mbox{and} \quad k_z = \sum _{j=0} ^\infty \ov{z} ^j k_j. $$
Adjoints of multipliers also have a natural convolution action on coefficient kernels, if $h \in \mr{Mult} (K , k)$, then, 
\be M_h ^* K_\ell = \sum _{i+j = \ell} k_i \ov{\hat{h} _j}. \label{convo} \ee We will say that a RKHS, $\cH (k)$, of analytic functions in $X = \D$ is a \emph{coefficient RKHS in} $\D$, if Taylor coefficient evaluations define bounded linear functionals on $\cH (k)$. In this case the positive coefficient kernel function $\hat{k}$ on $\N \cup \{ 0 \}$ is an example of a discrete or formal reproducing kernel in the sense of \cite{BV-formalRKHS}.

In this paper it will be useful to consider densely--defined multipliers between RKHS $\cH (K) , \cH (k)$ on $X$, which are not necessarily bounded. 
\begin{prop}[Multipliers are closeable] \label{closedmult}
Let $k, K$ be positive kernel functions on $X$, and let $h$ be a function on $X$ so that the linear operator $M_h : \nbdom M_h \subseteq \cH (k) \rightarrow \cH (K)$ has dense domain, $\nbdom M_h$. Then $M_h$ is closeable, and closed on its maximal domain, $\mr{Dom} _{\mr{max}} \, M_h := \{ g \in \cH (k) | \ h \cdot g \in \cH (K) \}$, $M_h ^* K_x = k_x \ov{h(x)}$, and $\bigvee _{x \in X} K_x$ is a core for $M_h ^*$, if $M_h$ is defined on its maximal domain.
\end{prop}
Recall that a linear operator with dense domain in a Hilbert space, $\cH$, is said to be \emph{closed} if its graph is a closed subspace of $\cH \oplus \cH$. Further recall that a dense set, $\scr{D} \subseteq \nbdom A$, contained in the domain of closed operator, $A$, is called a \emph{core} for $A$ if $A$ is equal to the \emph{closure} (minimal closed extension) of its restriction to $\scr{D}$. In general, given any two linear transformations $A, B$, we say that $B$ is an \emph{extension} of $A$ or that $A$ is a restriction of $B$, written $A \subseteq B$, if $\nbdom A \subseteq \nbdom B$ and $B | _{\nbdom A} = A$. Equivalently, the set of all pairs $(x, Ax)$, for $x \in \scr{D}$, is dense in the graph of $A$. Finally, $A$ is \emph{closeable} if it has a closed extension.
\begin{proof}
Define $\mr{Dom} _{\mr{max}} \, M_h$ to be the linear space of all $g \in \cH (k)$ so that $h \cdot g \in \cH (K)$. This is the largest domain on which $M_h$ makes sense. If $g_n \in \mr{Dom} _{\mr{max}} \, M_h$ is such that $g_n \rightarrow g$ and $M_h g_n \rightarrow f$, then since $\cH (k) , \cH (K)$ are RKHS, it necessarily follows that 
$$ g_n (x) \rightarrow g (x), \quad \mbox{and} \quad h(x) g_n (x) \rightarrow h(x) g(x) = f(x), \ x \in X. $$ This proves that $f = h \cdot g$, so that $g \in \mr{Dom} _{\mr{max}} \, M_h$ and $M_h$ is closed on $\mr{Dom} _{\mr{max}} \, M_h$. If $M_h$ is densely-defined on some other domain, $\nbdom M_h$, then $\nbdom M_h \subseteq \mr{Dom} _{\mr{max}} \, M_h$ by maximality, so that $M_h$ has a closed extension, and is hence closeable.

The fact that $\bigvee K_x$ is a core for $M_h ^*$ follows from the assumption that $M_h$ is defined (and closed) on its maximal domain. By maximality, $M_h$, with domain $\mr{Dom} _{\mr{max}} \, M_h$, has no non-trivial closed extensions which act as multiplication by $h$. Let $T_*$ be the closure of the restriction of $M_h ^*$ to $\bigvee K_x$. Then $T _* \subseteq M_h ^*$ is densely--defined and closed so that $M_h \subseteq T:=T_* ^*$, where $T_* ^*$, the adjoint of $T_*$ is necessarily closed so that $T ^* = T_*$. However,
$$ T^*  K_x = M_h ^* K_x = K_x \ov{h(x)}, $$ so that $T$ necessarily acts as multiplication by $h$ on its domain. By maximality, $\nbdom T = \nbdom _{max} M_h$ and $M_h =T$.
\end{proof}

\begin{remark} \label{coeffcore}
If $\cH (k)$ and $\cH (K)$ are Taylor coefficient RKHS in $\D$, then one can further show that the adjoint of any closed multiplication operator, $M_h : \cH (k) \rightarrow \cH (K)$ acts as a convolution operator on coefficient kernels, as in Equation (\ref{convo}), and the linear span of all Taylor coefficient kernels is also a core for $M_h ^*$.
\end{remark}

One can define a natural partial order on positive kernel functions on a fixed set, $X$. Namely, if $k$ and $K$ are two positive kernel functions on the same set, $X$, we write $k \leq K$, if $K-k$ is a positive kernel function on $X$. Notice that the identically zero kernel function is a positive kernel on $X$, so that $k \leq K$ can be equivalently written as $K- k \geq 0$. The following theorem of Aronszajn describes when one RKHS of functions on $X$ is boundedly contained in another in terms of this partial order, \cite[Section 7]{Aron-rkhs} \cite[Theorem 5.1]{Paulsen-rkhs}.

\begin{thm*}[Aronszajn's inclusion theorem]
Let $k, K$ be positive kernel functions on a set, $X$. Then $\cH (k) \subseteq \cH (K)$ and the norm of the embedding $\mr{e} : \cH (k ) \hookrightarrow \cH (K)$ is at most $t^2 >0$ if and only if $t^2 K \geq k$.
\end{thm*}

If $k$ and $K$ are both positive kernel functions on a set, $X$, it is immediate that $k+K$ is also a positive kernel function on $X$. The following `sums of kernels' theorem of Aronszajn describes the norm of $\cH (k +K )$ and the decomposition of this space in terms of $\cH (k)$ and $\cH (K)$ \cite{Aron-rkhs}, \cite[Theorem 5.4, Corollary 5.5]{Paulsen-rkhs}. Notice, in particular that $k , K \leq k + K$ as kernel functions so that $\cH (k)$ and $\cH (K)$ are contractively contained in $\cH (k + K )$, by the inclusion theorem.

\begin{thm*}[Aronszajn's sums of kernels theorem]
Let $k, K$ be positive kernel functions on a set, $X$. Then, $\cH (k+K) = \cH (k) + \cH (K)$ and 
$$ \| h \| ^2 _{\cH (k+K)} = \min \, \{ \| f \| ^2 _{\cH (k)} + \| g \| ^2 _{\cH (K)} | \ f + g = h \}. $$ 
In particular, $\cH (k+ K) = \cH (k) \oplus \cH (K)$ if and only if $\cH (k) \cap \cH (K) = \{ 0 \}$.
\end{thm*}

Observe that the sums of kernels theorem asserts that the algebraic sum $\cH (k+K) = \cH (k) + \cH (K)$ is a direct sum if and only if it is an orthogonal direct sum. More can be said about this decomposition and the structure of $\cH (k+K)$ using the theory of operator--range spaces of contractions and their complementary spaces in the sense of de Branges and Rovnyak \cite{dBR-ss}, \cite[Chapter 16]{FMHb}. Let $A \in \scr{L} (\cH , \cJ )$ be a bounded linear operator. The \emph{operator--range space of $A$}, $\scr{R} (A)$, is the Hilbert space obtained by equipping the range of $A$ with the inner product that makes $A$ a co-isometry onto its range. That is, $\scr{R} (A) = \nbran A \subseteq \cJ$, with inner product,
$$ \ip{Ax}{Ay}_A := \ip{x}{P_{\nbker A} ^\perp y}_{\cH}. $$
One can generally show that $\scr{R} (A) = \scr{R} ( \sqrt{A A^*} )$, \cite[Corollary 16.8]{FMHb}. If $A$ is a contraction, $\| A \| \leq 1$, then $\scr{R} (A) \subseteq \cJ$ is contractively contained in $\cJ$ in the sense that the embedding, $\mr{e} : \scr{R} (A) \hookrightarrow \cJ$ is a linear contraction. In this case, one can define the \emph{complementary space of $A$}, $\scr{R} ^c (A) := \scr{R} (\sqrt{I-AA^*} )$. The notion of complementary space was originally introduced in a more geometric way by de Branges and Rovnyak \cite{dBR-ss}. Namely, if $\cH$ is any Hilbert space and $\scr{R} \subseteq \cH$ is a Hilbert space which is contractively contained in $\cH$, then $\scr{R} = \scr{R} (\mr{j} )$, where $\mr{j} : \scr{R} \hookrightarrow \cH$ is the contractive embedding. L. de Branges and J. Rovnyak defined the complementary space, $\scr{R} ^c$ of $\scr{R}$ as the set of all $y \in \cH$ so that 
$$ \sup _{x \in \scr{R}} \left( \| y + x \| _{\cH} ^2 - \| x \| ^2 _{\scr{R}} \right) < + \infty. $$ One can prove that $\scr{R} ^c = \scr{R} ^c (\mr{j} )$ and that the above formula is equal to the norm of $y$ in $\scr{R} ^c (\mr{j})$, so that these two definitions coincide \cite[Chapter 16]{FMHb}. The following theorem summarizes several results in the theory of operator--range spaces, see \cite[Chapter 16]{FMHb}. 
\begin{thm}[Operator--range spaces of contractions] \label{oprancomp}
Let $A \in \scr{L} ( \cH , \cJ )$ be a contraction. If $\mr{e} : \scr{R} (A) \hookrightarrow \cJ$ and $\mr{j} : \scr{R} ^c (A) \hookrightarrow \cJ$ are the contractive embeddings, then
$$ \cJ = \scr{R} (A) + \scr{R} ^c (A). $$ For any $x = y + z \in \cJ$ so that $y \in \scr{R} (A)$ and $z \in \scr{R} ^c (A)$, the Pythagorean equality,
\be \| x \| ^2 _{\cJ} = \| y \| ^2 _{\scr{R} (A)} + \| z \| ^2 _{\scr{R} ^c (A)}, \label{compdecomp} \ee holds if and only if $y = \mr{e} \mr{e} ^* x$ and $z= \mr{j} \mr{j} ^*x$, so that, in particular, $I _\cJ = \mr{e} \mr{e} ^* + \mr{j} \mr{j} ^*$. As a vector space, the overlapping space is 
$$ \scr{R} (A) \cap \scr{R} ^c (A) = A \scr{R} ^c (A^*), $$ and $A : \scr{R} ^c (A^*) \rightarrow \scr{R} ^c (A)$ acts as a linear contraction. \\

Moreover, the following are equivalent:
\bi
    \item[\emph{(i)}] $A$ is a partial isometry,
    \item[\emph{(ii)}] $\scr{R} (A)$ and $\scr{R} ^c (A)$ are isometrically contained in $\cJ$ as orthgonal complemements, $\cJ = \scr{R} (A) \oplus \scr{R} ^c (A)$,
    \item[\emph{(iii)}] $\scr{R} (A) \cap \scr{R} ^c (A) = \{ 0 \}$.
\ei
\end{thm}
Observe that, as in Aronszajn's sums of kernels theorem, the algebraic sum $\cJ = \scr{R} (A) + \scr{R} ^c (A)$ is a direct sum if and only if it is an orthogonal direct sum. 

\begin{thm} \label{compRKHS}
Let $\cH (K)$ be a RKHS on a set, $X$. If $\cH (k)$ is another RKHS on $X$ which embeds, contractively, in $\cH (K)$, and $\mr{e} : \cH (k) \hookrightarrow \cH (K)$ is the contractive embedding, then $\cH (k) = \scr{R} (\mr{e} )$ and the complementary space, $\scr{R} ^c (\mr{e} )$, is the RKHS on $X$ with reproducing kernel $K-k$. 
\end{thm}
An embedding of RKHS, $\mr{e}: \cH (k) \hookrightarrow \cH (K)$, is necessarily injective. 
\begin{proof}
Let $\mr{e} : \cH (k) \hookrightarrow \cH (K)$ be the contractive embedding and consider the operator--range space of $\mr{e}$. Given any $g,h \in \cH (k)$, we have that 
$$ \ip{\mr{e} g}{\mr{e} h}_{\mr{e}} = \ip{g}{h}_{\cH (k)}, $$ since $\mr{e}$ is injective. Hence, for any $x \in X$,
\be \ip{\mr{e} k _x}{\mr{e}h}_{\mr{e}} = \ip{k _x}{h}_{\cH (k)} = h (x) = (\mr{e} h ) (x), \label{kerembed} \ee and it follows that $\scr{R} (\mr{e} ) = \cH (k)$. Indeed, equation (\ref{kerembed}) shows that $\scr{R} (\mr{e} )$ is a reproducing kernel Hilbert space on $X$ with point evaluation vectors $\wt{k} _x := \mr{e} k_x$ and that for any $x,y \in X$,
$$ \wt{k} (x,y) = \ip{\wt{k} _x}{\wt{k} _y}_\mr{e} = \ip{k_x}{k_y} = k(x,y), $$ so that $\scr{R} (\mr{e} ) = \cH (\wt{k} ) = \cH (k)$.
Now consider the complementary space, $\scr{R} ^c (\mr{e} )$, of $\cH (k)= \scr{R} (\mr{e} )$. Since this complementary space is contractively contained in $\cH (K)$, for any $\sqrt{I - \mr{e} \mr{e} ^*} g \in \scr{R} ^c (\mr{e} )$, 
\ba \ip{(I - \mr{e} \mr{e} ^*) K_x}{\sqrt{I - \mr{e} \mr{e} ^*} g}_{\scr{H}}
& = & \ip{\sqrt{I - \mr{e} \mr{e} ^*} K_x}{g}_{\cH (K)} \\
& = & ( \sqrt{I - \mr{e} \mr{e} ^*} g ) (x), \ea proving that $\scr{R} ^c (\mr{e})$ is also a RKHS on $X$ with point evaluation vectors $k ' _x := ( I - \mr{e} \mr{e} ^* ) K _x$. Hence, 
\ba k'(x,y) & = & \ip{k' _x}{k ' _y}_\scr{H} = \ip{( I - \mr{e} \mr{e} ^* ) K_x }{( I - \mr{e} \mr{e} ^* ) K_y}_\scr{H} \\
& = & \ip{K_x}{(I - \mr{e} \mr{e} ^*) K_y}_{\cH (K)} \\
& = & K (x,y) - k (x,y). \ea 
If $\mr{j} : \scr{R} ^c (\mr{e}) \hookrightarrow \scr{H} ^+ (\mu )$ is the contractive embedding, then observe that $\mr{j} \mr{j} ^* + \mr{e} \mr{e} ^* = I_{\cH (K)}$, so that 
$$ \scr{R} (\mr{j} ) = \scr{R} (\sqrt{\mr{j} \mr{j} ^*} ) = \scr{R} ^c (\mr{e} ). $$ 
\end{proof}
The previous theorem and Theorem \ref{oprancomp} provide additional information on the structure and decomposition of $\cH (k +K)$ in Aronszajn's sums of kernels theorem.
\begin{cor}
Let $k,K$ be positive kernel functions on a set, $X$ and let $\mr{e} : \cH (k) \hookrightarrow \cH (k+K)$ and $\mr{j} : \cH (K) \hookrightarrow \cH (k+K)$ be the contractive embeddings. Then we can identify $\cH (k)$ and $\cH (K)$ with the operator range spaces $\scr{R} (\mr{e} )$ and $\scr{R} (\mr{j})$, respectively. Moreover, $I_{\cH (k+K)} = \mr{e} \mr{e} ^* + \mr{j} \mr{j} ^*$ so that 
$\cH (K) = \scr{R} ^c (\mr{e} )$ is the complementary space of $\scr{R} (\mr{e} ) = \cH (k)$, and given any $h \in \cH (k + K)$, 
$$ \| h \|_{k+K} ^2 = \| \mr{e} ^* h \| _k ^2 + \| \mr{j} ^* h \| ^2 _K. $$ The intersection space, $\cH (k) \cap \cH (K)$ is equal to $\mr{e} \scr{R} ^c (\mr{e} ^*)$ and $\mr{j} \scr{R} ^c (\mr{j} ^* )$, and $\mr{e} : \scr{R} ^c (\mr{e} ^*) \rightarrow \scr{R} ^c (\mr{e} ) = \cH (K)$, $\mr{j} : \scr{R} ^c (\mr{j} ^* ) \rightarrow \scr{R} ^c (\mr{j})$ are contractions.
\end{cor}

Finally, as described in \cite{GO-rkhs} and \cite[Section 5]{MS-dBB}, we can define a pair of natural `lattice operations', $\vee$ and $\wedge$ on the set of all positive kernel functions on a fixed set, $X$. Given two positive kernel functions, $k$ and $K$, on $X$, let $k\vee K := k + K$, a positive kernel function on $X$. We can also construct a second RKHS on $X$ by defining 
$$ \mr{int} (k , K) := \cH (k ) \cap \cH (K), $$ equipped with the inner product
$$ \ip{g}{h}_{k \wedge K} := \ip{g}{h}_k + \ip{g}{h}_K. $$ It is not difficult to verify that $\mr{int} (k,K)$, equipped with this inner product is complete, and that point evaluation at any $x \in X$ defines a bounded linear functional on $\mr{int} (k, K)$, so that this is a RKHS, $\cH (k \wedge K)$, of functions on $X$. The following theorem describes a useful relationship between $\cH (k + K)$ and $\cH (k \wedge K)$ \cite[Theorem 5.2]{MS-dBB}, \cite{GO-rkhs}. 

\begin{thm*}[Sums and intersections of RKHS]
Let $k, K \geq 0$ be positive kernel functions on a set, $X$. Define two linear maps, $U_\vee$ and $U_\wedge$ of $\cH (k + K)$ and $\cH (k\wedge K)$, respectively, into $\cH (k) \oplus \cH (K)$ by 
$$ U_\vee (k +K )_x := k_x \oplus K_x, \ x \in X, \quad \mbox{and} \quad U_\wedge f := f \oplus -f, \ f \in \cH (k \wedge K). $$
Then, $U_\vee, U_\wedge$ both define isometries into $\cH (k) \oplus \cH ( K)$ with $\nbran U_\vee = \nbran U_\wedge ^\perp$, so that 
$$ \cH (k) \oplus \cH (K) = U_\vee \cH (k + K ) \oplus U_\wedge \cH (k \wedge K ). $$
The point evaluation vectors for $\cH (k \wedge K) = \cH (k) \cap \cH (K)$ are given by the formulas
$$ (k \wedge K) _x = \frac{1}{2} U_\wedge ^* ( K_x \oplus - k_x ) = U_\wedge ^* (K_x \oplus 0 ) = U_\wedge ^* ( 0 \oplus -k_x). $$
\end{thm*}

\subsection{Positive quadratic forms}

\label{ss:forms}

A quadratic or sesquilinear form, $\fq : \nbdom \fq \times \nbdom \fq \rightarrow \C$, with dense \emph{form domain}, $\nbdom \fq$ in a separable Hilbert space, $\cH$, is said to be positive semi-definite if $\fq (x,x) \geq 0$ for all $x \in \nbdom \fq$. Such a quadratic form is said to be \emph{closed}, if $\nbdom \fq$ is complete with respect to the norm induced by the inner product
$$ \ip{\cdot}{\cdot}_{\fq +\mf{id}} := \ip{\cdot}{\cdot}_\cH + \fq (\cdot , \cdot )$$ and $\fq$ is \emph{closeable} if it has a closed extension. We will let $\hat{\cH} (\fq)$ denote the Hilbert space completion of $\nbdom \fq$ with respect to this $\fq+\mf{id}-$inner product. Hence, $\fq$ is closed if and only if $\hat{\cH} (\fq) = \nbdom \fq$. If $\fq\geq 0$ is closeable, then its \emph{closure}, $\ov{\fq}$, is the minimal closed extension of $\fq$. By Kato, closed positive semi-definite forms obey an `unbounded version' of the Riesz Lemma \cite[Chapter VI, Theorem 2.1 and Theorem 2.23]{Kato}. Namely, $\fq \geq 0$ is closed if and only if there is a \emph{unique} self-adjoint, densely--defined and positive semi-definite operator, $A$, so that $\nbdom \fq = \nbdom \sqrt{A}$ and 
$$ \fq (x,x) = \fq_A (x,x) = \ip{\sqrt{A} x}{\sqrt{A} x}_{\cH}. $$
 Any self-adjoint operator is necessarily closed. Following Kato and Simon, we can define a partial order on densely--defined and positive semi-definite forms by $\fq_1 \leq \fq_2$ if
\bi
   \item[(i)] $\nbdom \fq_2 \subseteq \nbdom \fq_1$, and
   \item[(ii)] $\fq_1 (x,x) \leq \fq_2 (x,x)$ for all $x \in \nbdom \fq_2$. 
\ei
In particular, if $\fq_A$ and $\fq_B$ are the closed forms of the self-adjoint and positive semi-definite operators $A$ and $B$, we say that $A \leq B$ in the \emph{form sense} if $\fq_A \leq \fq_B$ as forms. This reduces to the usual L\"owner partial order on bounded, self-adjoint operators, if $A$ and $B$ are bounded. At first sight, it may seem strange that the `larger' form in the above definition may have a smaller domain. The following result of Kato provides some justification for this \cite[Chapter VI, Theorem 2.21]{Kato}, \cite[Proposition 1.1]{Simon-forms}.

\begin{lemma}[Kato] \label{resolvent}
Let $A,B \geq 0$ be self-adjoint and densely--defined in $\cH$. Then $A \leq B$ in the form sense if and only if 
$$ ( t I + A ) ^{-1} \leq (tI + B ) ^{-1}, $$ for any $t>0$.
\end{lemma}

Recall that if $A$ is a closed operator with dense domain, $\nbdom A \subseteq \cH$, that $\scr{D} \subseteq \nbdom A$ is called a core for $A$, if $A$ is equal to the closure of its restriction to $\scr{D}$. Similarly, if $\fq$ is a closed, densely--defined and positive semi-definite form, a (necessarily dense) set $\scr{D} \subseteq \nbdom \fq$ is called a \emph{form-core} for $\fq$, if $\scr{D}$ is dense in $\hat{\cH} (\fq)$. It is not difficult to verify that if $\fq=\fq_A$ is closed, then $\scr{D}$ is a form-core for $\fq$ if and only if $\scr{D}$ is a core for $\sqrt{A}$.  \\

\paragraph{Toeplitz forms.} The classical Hardy space, $H^2 = H^2 (\D )$, is the Hilbert space of square-summable Taylor series in the complex unit disk, equipped with the $\ell^2-$inner product of these coefficients. By results of Fatou, any element of $H^2$ has non-tangential boundary limits almost everywhere on the unit circle, $\partial \D$, with respect to normalized Lebesgue measure, $m$ \cite{Hoff}. Identifying any $h \in H^2$ with its boundary limits defines an isometric inclusion of $H^2$ into $L^2 = L^2 (m)$. Classically,  Toeplitz operators, $T$, on $H^2$, are defined as the compression of bounded multiplication operators on $L^2$ to $H^2$. Namely, $T = T_g := P_{H^2} M_g | _{H^2}$, and $\| T _g \| = \| M_g \| = \| g \| _\infty$ so that $g \in L^\infty$. A theorem of Brown and Halmos, \cite[Theorem 6]{BrownHalmos}, characterizes the Toeplitz operators as the set of all bounded operators, $T$, on $H^2$ which obey the simple algebraic condition,
$$ S^* T S = T, $$ where $S = M_z$, is the \emph{shift} on $H^2$, the isometry of multiplication by $z$. Under the boundary value identification of $H^2$ with the subspace $H^2 (m) \subseteq L^2 (m , \partial \D)$, the shift is identified with the isometry $S=Z^m = M^m _\zeta |_{H^2(m)}$. 

Recall, as described in the Outline, given a positive, finite and regular Borel measure, $\mu$, on $\partial \D$, we can associate to $\mu$ the densely--defined and positive semi-definite quadratic form, $\fq _\mu$, with dense form domain, $\nbdom \fq _\mu := A (\D )$ in $H^2 = H^2 (m )$, where $m$ is normalized Lebesgue measure.
This positive form, $\fq _\mu$, is an example of a \emph{Toeplitz form}, as studied by Grenander and Szeg\"o in \cite{Szego}. Namely, $\nbdom \fq_\mu = A (\D )$ obeys $S \nbdom \fq_\mu \subseteq \nbdom \fq_\mu$, and 
$$ \fq_\mu (S x , Sy) = \fq_\mu (x,y); \quad \quad x,y \in A (\D ) = \nbdom \fq_\mu.$$ If $\fq_\mu$ is closeable so that $\ov{\fq _\mu} = \fq_T$ for a closed, self-adjoint $T \geq 0$, then by Kato's unbounded Riesz lemma we have that $S^* T S = T$, and our results will show that
$$ \fq_\mu (x,y) = \ip{M_{\sqrt{f}} x}{M_{\sqrt{f}} y}_{L^2}; \quad \quad f = \frac{d\mu}{dm} \in L^1, \ x,y \in A (\D ), $$ see Theorem \ref{RKac}. Hence, $T = T_f = P_{H^2} M_f | _{H^2}$ is a closed, potentially unbounded Toeplitz operator with symbol $f$, in this `quadratic form sense'. In particular, if $T\geq 0$ is bounded, which happens if and only if $\mu \leq t^2 m$ for some $t >0$, then by the Riesz representation lemma, $S^* T S = T$, so that $T$ is a bounded Toeplitz operator by Brown--Halmos, in which case $T = T_f$ for $f = \frac{d\mu}{dm} \in L^\infty$ and $\| f \| _\infty \leq t$, by Theorem \ref{RKdominate} and Corollary \ref{bddRN}.

\section{Spaces of Cauchy transforms}
\label{CTsec}

Let $\mu$ be a positive, finite and regular Borel measure on the complex unit circle. Recall that given any $h \in H^2 (\mu ) = \C [ \zeta ] ^{- \| \cdot \| _{L^2 (\mu)}}$, its $\mu-$Cauchy transform is the function
\be (\scr{C} _\mu h) (z) := \int _{\partial \D } h (\zeta) \, \frac{1}{1-z \ov{\zeta}} \mu (d\zeta) = \ip{k_z}{h}_{L^2 (\mu )}; \quad k_z (\zeta ) := (1 - \ov{z} \zeta ) ^{-1}. \ee We will call the functions $k_z$, $z \in \D$, \emph{Szeg\"o kernel vectors}.

Recall that $A (\D )$ denotes the disk algebra, the unital Banach algebra of analytic functions in $\D$ which extend continuously to the boundary, $\partial \D$. Since the analytic polynomials are supremum-norm dense in $A (\D )$, viewed as a subspace of the continuous functions, $\scr{C} (\partial \D )$, on the circle and $H^2 (\mu ) = A (\D ) ^{-\| \cdot \| _{L^2 (\mu )}}$, it follows that $H^2 (\mu ) = \C [\zeta ] ^{-\| \cdot \| _{L^2 (\mu )}} = A (\D ) ^{-\| \cdot \| _{L^2 (\mu )}}$. Similarly, $\scr{K} _\D := \bigvee _{z \in \D} k_z$ is also supremum-norm dense in $A (\D ) \subseteq \scr{C} (\partial \D )$, so that this subspace is also dense in $H^2 (\mu )$. 

\begin{lemma} \label{lemholo}
The $\mu-$Cauchy transform of any $h \in H^2 (\mu )$ is holomorphic in $\D$.
\end{lemma}
\begin{proof}
Given any $z \in \D$, since $(\scr{C} _\mu h) (z) := \ip{k_z}{h}_{L^2 (\mu )}$, consider the Leibniz difference quotient
$$ \lim _{\eps \rightarrow 0} \frac{(\scr{C} _\mu h) (z +\eps) -(\scr{C} _\mu h) (z)}{\eps} = \lim _{\eps \rightarrow 0} \eps ^{-1} \ip{k_{z+\eps} - k_z}{h}_{L^2 (\mu )}. $$ (Here, recall that our inner products are conjugate linear in their first argument.) This limit will exist and $\scr{C} _\mu h$ will be holomorphic, if and only if the limit of $\ov{\eps} ^{-1} (k_{z+\eps} - k_z)$ exists in $H^2 (\mu )$. This limit exists in supremum norm on the circle (and so belongs to $A (\D )$), and so it certainly exists in the $L^2 (\mu)-$norm by Cauchy--Schwarz. Indeed,
\ba \lim _{\eps \rightarrow 0} \ov{\frac{k_{z+\eps} (\zeta ) - k_z (\zeta )}{\ov{\eps}}} & = & \lim \frac{1 - \ov{\zeta} z - (1 - \ov{\zeta} z - \ov{\zeta} \eps)}{\eps (1 - \ov{\zeta} z) (1-\ov{\zeta} (z +\eps ))} \\
& = & \frac{\ov{\zeta}}{(1-\ov{\zeta} z) ^2}, \ea and this limit is continuous on $\partial \D$ for any fixed $z \in \D$. Hence $\scr{C} _\mu h \in \scr{O} (\D)$ for any $h \in H^2 (\mu )$.
\end{proof}

Let $\scr{H} ^+ (\mu ) := \scr{C} _\mu H^2 (\mu ) \subseteq \scr{O} (\D)$ be the complex vector space of $\mu-$Cauchy transforms equipped with the inner product,
$$ \ip{\scr{C} _\mu g}{\scr{C} _\mu h}_\mu := \ip{g}{h}_{L^2 (\mu )}. $$ 
\begin{lemma} \label{lemCT}
The space of $\mu-$Cauchy transforms, $\scr{H} ^+ (\mu )$, is a RKHS of analytic functions in $\D$ with reproducing kernel
\ba k^\mu (z,w) & := & \int _{\partial \D} \frac{1}{1-z \ov{\zeta}} \frac{1}{1-\ov{w} \zeta} \mu (d\zeta ) = \ip{k_z}{k_w}_{L^2 (\mu )} \\
& = & \frac{1}{2} \frac{H_\mu (z) + \ov{H_\mu (w)}}{1-z \ov{w}}, \ea 
where $H_\mu (z) = \int _{\partial \D} \frac{1 + z \ov{\zeta}}{1-z\ov{\zeta}} \mu (d\zeta)$ is the Herglotz--Riesz transform of $\mu$.
\end{lemma}
By construction, $\scr{C} _\mu$ is an isometry of $H^2 (\mu )$ onto $\scr{H} ^+ (\mu )$.
\begin{proof}
To show that this inner product is well-defined, we need to check that $\scr{C} _\mu h \equiv 0$ in the disk implies that $h=0$ in $H^2 (\mu)$. Indeed,
$$ (\scr{C} _\mu h) (z) = \ip{k_z}{h}_{H^2 (\mu )}, $$ and since $\bigvee k_z$ is dense in $A (\D)$, the linear span of the Szeg\"o kernels is also dense in $H^2 (\mu)$ as described above. Hence this vanishes for all $z \in \D$ if and only if $h=0$. 

By definition, for any $z \in \D$,
$$ (\scr{C} _\mu h) (z) = \ip{k_z}{h}_{H^2 (\mu )} = \ip{\scr{C} _\mu k_z}{\scr{C} _\mu h}_\mu, $$ so that $\scr{H} ^+ (\mu )$ is a RKHS in $\D$ with kernel vectors $k ^\mu _z := \scr{C} _\mu k_z$ and reproducing kernel
$$ k^\mu (z,w) := \ip{\scr{C} _\mu k_z}{\scr{C} _\mu k_w}_{L^2 (\mu )} = \int _{\partial \D}  \frac{1}{1 -z \ov{\zeta}} \frac{1}{1-\ov{w} \zeta} \mu (d\zeta ). $$ Finally, 
\ba H_\mu (z) + \ov{H_\mu (w)} & = & \int _{\partial \D} \frac{1 +z\ov{\zeta}}{1-z\ov{\zeta}} \mu (d\zeta ) + \int _{\partial \D} \frac{1 +\ov{w} \zeta}{1-\ov{w}\zeta} \mu (d\zeta ) \\
& = & 2 (1 - z \ov{w} ) \int _{\partial \D} \frac{1}{1-z \ov{\zeta}} \frac{1}{1-\ov{w} \zeta} \mu (d\zeta ), \ea establishing the second formula.
\end{proof}

\begin{eg}[Hardy space]
If $\mu =m$ is normalized Lebesgue measure, then,
\ba H_m (z)  & := & \int _{\partial \D} \frac{1 + z \ov{\zeta}}{1 -z \ov{\zeta}} m (d\zeta ) \\
& = & \int _{\partial \D} \frac{2}{1 -z \ov{\zeta}} m (d\zeta ) - m (\partial \D) \\
& = & \sum _{j=0} ^\infty z^j  \int _{\partial \D} \ov{\zeta} ^j m (d\zeta ) - \int _{\partial \D} m (d\zeta ) \\
& = & \sum _{j=0} ^\infty z^j  \frac{1}{2\pi} \int _0  ^{2\pi} e^{-ij\theta} d\theta - \frac{1}{2\pi} \int _0 ^{2\pi} d\theta = 2 \delta _{j,0} - 1 =1. \ea
It follows that $b_m := \frac{H_m -1}{H_m +1} \equiv 0$, so that $m = \mu _0$ is the Clark measure of the identically $0$ function. 
Moreover, 
$$ k^m (z,w) = \frac{1}{2} \frac{H_m (z) + \ov{H_m (w)}}{1-z\ov{w}} = \frac{1}{1-z\ov{w}} = k(z,w), $$ is the \emph{Szeg\"o kernel}. This is the reproducing kernel for the classical Hardy space $H^2 = H^2 (\D )$, of square--summable Taylor series in the complex unit disk, equipped with the $\ell ^2-$inner product of the Taylor coefficients. That is, $\scr{H} ^+ (m) = H^2$.
\end{eg}

Since any $h := \scr{C} _\mu g \in \scr{H} ^+ (\mu )$ is holomorphic in the open unit disk, its Taylor series at $0$ has radius of convergence at least one,
$$ h(z) = \sum _{j=0} ^\infty \hat{h} _j z^j. $$ Moreover, expanding $k_z (\zeta)$ in a convergent geometric sum,
\ba h(z) = \sum _{j=0} ^\infty \hat{h} _j z^j & = & \ip{k_z}{g}_{L^2 (\mu ) } = \sum _{j=0} ^\infty z^j \ip{\zeta ^j}{g}_{L^2 (\mu )} \\
& = & \sum _{j=0} ^\infty z^j \ip{\scr{C} _\mu \zeta ^j}{h}_\mu, \ea and it follows that the Taylor coefficients are given by 
$$ \hat{h} _j = \ip{\scr{C} _\mu \zeta ^j}{h}_\mu; \quad \quad j \in \N \cup \{ 0 \}. $$ That is, for any $j \in \N \cup \{ 0 \}$, the linear functionals $\ell _j (h) = \hat{h} _j$ are bounded on $\scr{H} ^+ (\mu )$ and are implemented by inner products against the Taylor coefficient kernel vectors $k ^\mu _j := \scr{C} _\mu \zeta ^j$. Hence $\scr{H} ^+ (\mu )$ is a Taylor coefficient RKHS in $\D$ with coefficient reproducing kernel, $\hat{k} ^\mu (i,j)$, on the set $\N \cup \{ 0 \}$, 
$$ \hat{k} ^\mu (i,j) := \ip{k^\mu _i}{k^\mu _j}_{\mu}, $$ and $\hat{k} ^\mu$ is then a positive kernel function on $\N \cup \{ 0 \}$.

Given a positive measure $\mu$, let $V_\mu := \scr{C} _\mu Z^\mu \scr{C} _\mu ^*$. This is an isometry on $\scr{H} ^+ (\mu )$ that will play a central role in our analysis. This operator has a natural action on kernel vectors:
\be V _\mu k^\mu _z \ov{z} = k^\mu _z - k^\mu _0, \quad \mbox{and} \quad V_\mu ^* (k^\mu _z - k^\mu _0) = k^\mu _z \ov{z}. \ee In particular,
$$ \bigvee _{z \in \D} (k_z ^\mu - k_0 ^\mu ) = \nbran V_\mu, $$ where, here, $\bigvee$ denotes closed linear span. It is easy to check that a function, $h \in \scr{H} ^+ (\mu )$, is orthogonal to $\nbran V_\mu$ if and only if $h = c 1$, $c \in \C$, is constant in the disk. Hence the following statements are equivalent:
\bi
    \item[(i)] $\mu$ is extreme,
    \item[(ii)] $H^2 (\mu ) = L^2 (\mu )$,
    \item[(iii)] $Z^\mu = M^\mu _\zeta | _{H^2 (\mu )}$ and hence $V _\mu$ is unitary,
    \item[(iv)] $\scr{H} ^+ (\mu )$ does not contain the constant functions.
\ei

\begin{lemma} \label{bwshift}
Given any finite, positive and regular Borel measure, $\mu$, on $\partial \D$, the co-isometry $V_\mu ^*$, acts as a backward  shift on $\scr{H} ^+ (\mu )$, \emph{i.e.} if $h \in \scr{H} ^+ (\mu)$, $h(z) = \sum _{j=0} ^\infty \hat{h} _j z^j$, then
$$ (V_\mu ^* h) (z) = \sum _{j=0} ^\infty \hat{h}_{j+1} z^j = \frac{h(z) - h(0)}{z}. $$
Given any $h \in \scr{H} ^+ (\mu )$,
$$ (V_\mu h) (z) = z h(z) + (V_\mu h) (0) 1. $$
\end{lemma}
Given any $h$ in the classical Hardy space, $H^2 = \scr{H} ^+ (m)$, one can check that $S := V_m = M_z$ is the isometry of multiplication by the independent variable, $z$, on $H^2$, the \emph{shift}. In this case, adjoint of $S$ is called the backward shift and acts as 
$$ (S^*h) (z) = \frac{h(z) - h(0)}{z}. $$ It is straightforward to verify that if $h \in H^2$ has Taylor series $h(z) = \sum \hat{h} _j z^j$, then
$(S^* h) (z) = \sum _{j=0} ^\infty \hat{h} _{j+1} z^j$. This motivates the terminology `backward shift' in the above lemma statement. This lemma is easily verified and we omit the proof. 

\section{Absolute continuity in the reproducing kernel sense}
\label{sec:RKac}
 Recall that given positive measures $\mu$ and $\la$, we say that $\mu$ is \emph{dominated} by $\la$ if there is a $t>0$ so that $\mu \leq t^2 \la$, and we say that $\mu$ is \emph{reproducing kernel} or \emph{RK-dominated} by $\la$, if $\scr{H} ^+ (\mu ) \subseteq \scr{H} ^+ (\la )$ and there is a $t>0$ so that the norm of the embedding $\mr{e} _{\mu, \la} : \scr{H} ^+ (\mu ) \hookrightarrow \scr{H} ^+ (\la )$ is at most $t$, written $\mu \leq _{RK} t^2 \la$. We will begin this section by showing that these two definitions of domination are equivalent.

\begin{thm} \label{RKdominate}
Given positive, finite and regular Borel measures $\mu, \la$ on the unit circle, $\mu \leq t^2 \la$ for some $t>0$ if and only if $\mu \leq _{RK} t^2 \la$. 
\end{thm}
\begin{proof}
\noindent (\emph{Necessity.}) If $\mu \leq t^2 \la$, then $\ga := t^2 \la - \mu$ is a positive measure and 
\ba t^2 k^\la (z,w) - k^\mu (z,w) & = & t^2 \int _{\partial \D} \frac{1}{1-z\ov{\zeta}}\frac{1}{1-\ov{w}\zeta} \la (d\zeta ) - \int _{\partial \D} \frac{1}{1-z\ov{\zeta}}\frac{1}{1-\ov{w}\zeta} \mu (d\zeta ) \\
& = & \int _{\partial \D} \frac{1}{1-z\ov{\zeta}}\frac{1}{1-\ov{w}\zeta} \ga (d\zeta ) = k^\ga (z,w). \ea 
It follows that $t^2 k^\la - k^\mu = k^\ga$ is a positive kernel so that $k^\mu \leq t^2 k^\la$. \\

\noindent \emph{First proof of sufficiency.} Conversely, suppose that $K:= t^2 k ^\la - k^\mu \geq 0$ is a positive kernel. View the analytic polynomials, $\C [ \zeta ]$, as a dense subspace of the disk algebra $A (\D )$, embedded isometrically in the Banach space $\scr{C} (\partial \D )$. For any finite, positive and regular Borel measure on the complex unit circle, $\mu$, we define the positive linear functional, $\hat{\mu}$ on $\scr{C} (\partial \D)$ by 
$$ \hat{\mu} (f) := \int _{\partial \D} f d\mu. $$ (The map $\mu \mapsto \hat{\mu}$ is a bijection, by the Riesz--Markov theorem.) We then define a bounded linear functional, $\ell_K$ on $\scr{C} (\partial \D)$ by $\ell _K := t^2 \hat{\la} - \hat{\mu}$. 

By \mbox{Weierstra\ss}  approximation, $\C [ \zeta ] + \ov{\C [ \zeta ] }$ is supremum-norm dense in the continuous functions, $\scr{C} (\partial \D )$. Since the Fej\'er kernel is positive semi-definite, the partial Ces\`aro sums of any positive semi-definite $f \in \scr{C} (\partial \D )$ will be a positive trigonometric polynomial, \emph{i.e.} a positive semi-definite element of $\C [\zeta ] + \ov{\C [\zeta ]}$ and, by Fourier theory, it follows that the positive cone of $\C [ \zeta ] + \ov{\C [ \zeta ] }$ is supremum norm-dense in the positive cone of $\scr{C} (\partial \D)$. Moreover, by the Fej\'er--Riesz theorem, any positive trigonometric polynomial, $p + \ov{q} \geq 0$, on $\partial \D$ factors as $|g| ^2$ for an analytic $g \in \C [\zeta ]$ (and necessarily, $p =q$, $\mr{deg} (p) = \mr{deg} (g)$). Hence, to check that $\ell _K$ is a positive linear functional on $\scr{C} (\partial \D )$, it suffices to check that $\ell _K ( p + \ov{p} ) \geq 0$ for any $p + \ov{p} = |g| ^2 \geq 0$, $p, g \in \C [ \zeta ]$. If $p = \sum _{j=0} ^n \hat{p} _j \zeta ^j$ and $g = \sum _{j=0} ^n \hat{g} _j \zeta ^j$, then by construction
$$ \ip{\scr{C} _\mu g}{\scr{C} _\mu g}_\mu = \ip{g}{g}_{H^2 (\mu) } = \ip{p}{1}_{H^2 (\mu)} + \ip{1}{p}_{H^2 (\mu )} = \ip{\scr{C} _\mu p}{\scr{C} _\mu 1}_\mu + \ip{\scr{C} _\mu 1}{\scr{C} _\mu p}_{\mu}. $$
Since
$$ \scr{C} _\mu p = \sum _{j=0} ^n \hat{p} _j k ^\mu _j,$$ where the $k^\mu _j$, $j \in \N \cup \{ 0 \}$ are the Taylor coefficient evaluation vectors, and since similar formulas hold for $\la$, we obtain that 
$$ \ell _K (|g| ^2) = \int _{\partial \D} |g (\zeta ) | ^2 (t^2 \la (d\zeta ) - \mu (d\zeta ) )  =  \sum _{i,j=0} ^n \ov{\hat{g} _i} \hat{g} _j \ip{K_i}{K_j}_{\cH (K)}, $$ where the $K_i$, $i \in \N \cup \{ 0 \}$ are the Taylor coefficient evaluation vectors in $\cH (K)$. Namely, $\cH (K)$ is also a Taylor coefficient RKHS in $\D$ so that $\hat{K} (i , j) := \ip{K_i}{K_j}_{\cH (K)}$ defines a positive kernel function on the set $\N \cup \{ 0 \}$. It follows that 
$$ \int _{\partial \D} |g (\zeta ) | ^2 (t^2 \la (d\zeta ) - \mu (d\zeta ) ) \geq 0, $$ for any $g \in \C [ \zeta]$, or, equivalently, 
$$ \ell _K (p + \ov{p}) = \int _{\partial \D} (p + \ov{p} ) (t^2 d\la - d\mu ) \geq 0, $$ for any positive semi-definite $p + \ov{p} \in \C [\zeta] + \ov{\C [ \zeta ]}$. By density of the positive cone of $\C [\zeta ] + \ov{\C [\zeta ] }$ in the positive cone of the continuous functions, it follows that $\ell _K$ is a bounded, positive linear functional on $\scr{C} (\partial \D )$, with norm $\| \ell _K \| = \ell _K (1) = t^2 \la (\partial \D ) - \mu (\partial \D ) = K( 0, 0) \geq 0$. By the Riesz--Markov theorem, there is then a unique, positive, finite and regular Borel measure, $\ga$, on $\partial \D$, so that 
$$ \ell _K (f) = \int _{\partial \D} f(\zeta ) \ga (d\zeta ), $$ for any $f \in \scr{C} (\partial \D )$, \emph{i.e.} $\ell _K = \hat{\ga}$, and we conclude that $\ga = t^2 \la - \mu \geq 0$ so that $t^2 \la \geq \mu$. \\

\noindent \emph{Second proof of sufficiency.} If $t^2 k^\la \geq k^\mu$, then by Aronszajn's inclusion theorem, $\scr{H} ^+ (\mu ) \subseteq \scr{H} ^+ (\la )$ and the norm of the embedding $\mr{e} _{\mu , \la } : \scr{H} ^+ (\mu ) \hookrightarrow \scr{H} ^+ (\la )$ is at most $t>0$. Observe that $\mr{e} := \mr{e} _{\mu , \la}$ acts trivially as a multiplier by the constant function $1$, so that 
$$ \mr{e} ^* k ^\la _z = k^\mu _z, \quad \mbox{and} \quad \mr{e} ^* k^\la _j = k^\mu _j, $$ for any $z \in \D$ and $j \in \N \cup \{ 0 \}$. Hence,
\ba \mr{e} ^* V^\la k^\la _z \ov{z} & = & \mr{e} ^* ( k^\la _z - k ^\la _0 ) \\ 
& = & k^\mu _z - k ^\mu _0 = V^\mu k^\mu _z \ov{z} \\
& = & V^\mu \mr{e} ^* k^\la _z \ov{z}, \ea so that $\mr{e} ^*$ intertwines $V^\la$ with $V^\mu$, $\mr{e} ^* V^\la = V^\mu \mr{e} ^*$. Setting $E := \scr{C} _\mu ^* \mr{e} ^* \scr{C} _\la$, we see that for any monomial,
\ba E \zeta ^j & = & \scr{C} _\mu ^* \mr{e} ^* k^\la _j \\
& = & \scr{C} _\mu ^* k^\mu _j = \zeta ^j \in H^2 (\mu ), \ea so that $E : H^2 (\la ) \hookrightarrow H^2 (\mu)$ obeys $Ep =p$ for any $p \in \C [\zeta ] \subseteq H^2 (\la )$. In particular, $E Z^\la = Z^\mu E$. At this point one could argue using the Riesz--Markov theorem as above, however, here is an alternative argument. Since $Z^\la$ and $Z^\mu$ are contractions (they are isometries), we can apply the intertwining version of the commutant lifting theorem \cite[Corollary 5.9]{Paulsen} to conclude that $E$ can be `lifted' to a bounded operator $\hat{E} : L^2 (\la ) \rightarrow L^2 (\mu )$, with norm $\| \hat{E} \| = \| E \|$, so that $\hat{E} M^\la _\zeta = M^\mu _\zeta \hat{E}$, and $P _{H^2 (\mu )} \hat{E} | _{H ^2 (\la )} = E$. Setting $\hat{T} := \hat{E}^* \hat{E} \in \scr{L} (L ^2 (\la ) )$, $\hat{T} \geq 0$ is a positive semi-definite $M^\la _\zeta-$\emph{Toeplitz operator} in the sense that 
$$ M^{\la *} _\zeta \hat{T} M^\la _\zeta = \hat{E} ^* M^{\mu *} _\zeta M^\mu _\zeta \hat{E} = \hat{T}, $$ since $M ^\la _\zeta, M^\mu _\zeta$ are isometries. Since, $M^\la _\zeta$ (and $M^\mu _\zeta$) is in fact unitary, it follows that $\hat{T}$ commutes with $M^\la _\zeta$. Since $\C [\zeta ] + \ov{\C [\zeta ]} \subseteq \nbran M^\la _\zeta$ is a dense set in $L^2 (\la )$, a simple argument shows that $\hat{T}$ acts as multiplication by $f:=T1 \in L^2 (\la)$. However, since $T = M_f$ is a bounded and positive semi-definite operator, it is easy to check that $t ^2 \geq \| \hat{T} \| = \| f \| _{\infty }$, $\| \hat{T} \| = \| \hat{E} \| ^2 = \| E \| ^2 = \| T \|$, and $f \geq 0$ $\la-$a.e. Finally, one can also check that $T = P_{H^2 (\la )} \hat{T} | _{H^2 (\la)}$ is $Z^\la-$Toeplitz, \emph{i.e.} $Z^{\la *} T Z^\la =T$. In conclusion, for any $p, q \in \C [\zeta ]$,
\ba \int _{\partial \D} \ov{p (\zeta)} q (\zeta ) \mu (d\zeta ) & = & \ip{p}{q}_{H^2 (\mu )}  =  \ip{Ep}{Eq}_{H^2 (\mu )} \\
& = & \ip{\sqrt{T} p}{\sqrt{T}q}_{H^2 (\la )} = \ip{p}{\hat{T} q}_{L^2 (\la )} \\ 
& = & \ip{p}{M_f q}_{L ^2 (\la )} = \int _{\partial \D} \ov{p (\zeta )} q (\zeta) f(\zeta ) \la (d\zeta ). \ea 
This formula extends to elements of the form $g= \ov{p} + q \in \C [\zeta ] + \ov{\C [\zeta ]}$ since 
$$ \ip{\ov{p} _1 + q_1}{\ov{p_2} + q_2}_{L^2 (\mu )} = \ip{q_1}{q_2}_{H^2 (\mu )} + \ip{q_1 p_2}{1}_{H^2 (\mu)} + \ip{p_2}{p_1}_{H^2 (\mu )} + \ip{1}{p_1 q_2}_{H^2 (\mu )}.$$ Again, by \mbox{Weierstra\ss} approximation, since $\C [\zeta ] + \ov{\C [\zeta ]}$ is supremum-norm dense in $\scr{C} (\partial \D )$, which is in turn dense in $L^2 (\la )$ and $L^2 (\mu )$, it follows that for any $g, h \in L^2 (\la )$,
$$ \int _{\partial \D} \ov{g (\zeta )} h (\zeta ) \mu (d\zeta ) = \int _{\partial \D} \ov{g (\zeta )} h (\zeta ) f(\zeta ) \la (d\zeta ), $$ where $f \geq 0$, $\la-$a.e. and $\| f \| _\infty \leq t^2$. We conclude that $\mu \leq t^2 \la$ and that
$$ f = \frac{\mu (d\zeta )}{\la (d\zeta )}, $$ is the (bounded) Radon--Nikodym derivative of $\mu$ with respect to $\la$. 
\end{proof}

\begin{defn} \label{Toepdef}
Let $T \in \scr{L} (\cH )$ be a bounded operator and let $V$ be an isometry on $\cH$. We say that $T$ is $V-$\emph{Toeplitz} if
$$ V^* T V = T. $$ 

If $\fq \geq 0$ is a positive semi-definite quadratic form with dense form domain, $\nbdom \fq \subseteq \cH$, we say that $\fq$ is $V-$\emph{Toeplitz} if $\nbdom \fq$ is $V-$invariant and 
$$ \fq (Vg , Vh) = \fq (g,h); \quad \quad g,h \in \nbdom \fq. $$
In particular, if $T\geq 0$ is a positive semi-definite, self-adjoint and densely--defined operator in $\cH$, we say that $T$ is $V-$\emph{Toeplitz} if the closed, positive semi-definite form it generates, 
$$ \fq _T (x,y) := \ip{\sqrt{T} x}{\sqrt{T} y}_\cH; \quad \quad x,y \in \nbdom \fq _T = \nbdom \sqrt{T}, $$ is $V-$Toeplitz. If $T\geq 0$ is bounded, this latter definition reduces to the definition of a bounded, positive semi-definite $V-$Toeplitz operator.
\end{defn}

\begin{cor} \label{bddRN}
Let $\mu, \la$ be positive, finite and regular Borel measures on $\partial \D$ so that $\mu \leq t^2 \la$. In this case, 
$$ E_{\mu , \la} := \scr{C} _\mu ^* \mr{e} _{\mu , \la } ^* \scr{C} _\la : H^2 (\la ) \hookrightarrow H^2 (\mu),$$ is a \emph{co-embedding} in the sense that $E _{\mu ,\la} p = p$ for any $p \in \C [\zeta ]$. Moreover,
$$ \mr{e} _{\mu , \la} ^* k ^\la _z = k^\mu _z, \quad \mr{e} _{\mu , \la} ^* k^\la _j = k^\mu _j, \quad \quad \forall \ z \in \D, \ j \in \N \cup \{ 0 \},$$ $\mr{e} _{\mu ,\la } ^* V_\la = V^\mu \mr{e} _{\mu , \la } ^*$ and equivalently $E_{\mu , \la} Z ^\la = Z^\mu E_{\mu , \la}$. If $T_\mu := E _{\mu , \la} ^* E _{\mu ,\la}$, then $T_\mu$ is $Z^\la-$Toeplitz and $T$ is the compression of a bounded multiplication operator, $T = P_{H^2 (\la )} M_f |_{H^2 (\la)}$, where $f \geq 0$ $\la-a.e.$, $\| f \| _{\infty} \leq t^2$, and $f = \frac{\mu (d\zeta) }{\la (d\zeta)}$ is the Radon--Nikodym derivative of $\mu$ with respect to $\la$.
\end{cor}

\begin{remark}
While the co-embedding, $E_{\mu, \la} : H^2 (\la ) \hookrightarrow H^2 (\mu)$ always has dense range, it may have non-trivial kernel. For example, if $\la$ is the sum of two Dirac point masses at distinct points $\zeta , \xi \in \partial \D$, $\mu$ is the point mass at $\zeta$, then  $\mu \leq \la$ and if $p$ is any polynomial vanishing at $\zeta$, then $E_{\mu , \la} p =0 \in H^2 (\mu )$. To be precise, $H^2 (\mu )$ is the closure of the disk algebra, $A (\D )$, or the polynomials, $\C [\zeta ]$, in the $L^2 (\mu )-$norm, so that if $p \in \C [\zeta ]$ or $a \in A (\D )$ vanishes on the support of $\mu$, then $p=0=a$ in $H^2 (\mu )$.
\end{remark}

More generally, absolute continuity of positive measures can also be described in terms of their spaces of Cauchy transforms. It is a straightforward exercise, using the Radon--Nikodym formula, to show that $\mu$ is absolutely continuous with respect to $\la$, if and only if one can construct a monotonically increasing sequence of positive measures, $\mu _n \geq 0$, so that $\mu _n \leq \mu$ for all $n$, the $\mu _n \uparrow \mu$ increase monotonically to $\mu$, and there is a sequence of positive constants, $t_n >0$ so that $\mu _n \leq t_n ^2 \la$. Indeed, this can be readily established by taking the `join' or point-wise maxima of $\frac{d\mu }{d\la}$ and the constant functions $t_n ^2 \cdot 1$. Since $\mu _n \leq \mu$ for all $n$, Aronszajn's inclusion theorem implies that $\scr{H} ^+ (\mu _n ) \subseteq \scr{H} ^+ (\mu )$ and that the embeddings $\mr{e} _n : \scr{H} ^+ (\mu _n ) \hookrightarrow \scr{H} ^+ (\mu )$ are all contractive. Moreover, and again by Aronszajn's inclusion theorem, each $\scr{H} ^+ (\mu _n ) \subseteq \scr{H} ^+ (\la )$ is boundedly contained in $\scr{H} ^+ (\la )$ so that 
\be \scr{H} ^+ (\mu _n ) \subseteq  \scr{H} ^+ (\mu ) \cap \scr{H} ^+ (\la )=:\mr{int} (\mu , \la ). \ee

\begin{prop} \label{necessary}
If $\mu \ll \la$ then the intersection space, $\mr{int} (\mu , \la ) = \scr{H} ^+ (\mu ) \cap \scr{H} ^+ (\la )$ is dense in $\scr{H} ^+ (\mu )$.
\end{prop}
\begin{proof}
We have that for all $n$, $0 \leq \mu _n \leq \mu$ and $\mu _n \uparrow \mu$. Moreover, $\scr{H} ^+ (\mu _n ) \subseteq \mr{int} (\mu , \la )$ for all $n$. If $1 \geq g_n \geq 0, \mu-a.e.$ are the Radon--Nikodym derivatives of the $\mu _n$ with respect to $\mu$, and $p \in \C [\zeta ]$, let $\scr{C} _n := \scr{C} _{\mu _n}$ and let $\mr{e} _n : \scr{H} ^+ (\mu _n ) \hookrightarrow \scr{H} ^+ (\mu )$. Then, 
\ba \| \scr{C} _\mu p - \mr{e} _n \scr{C} _n p \| ^2 _\mu & = & \| \scr{C} _\mu p \| ^2 _\mu -2 \nbre \ip{\scr{C} _\mu p}{\mr{e} _n \scr{C} _n p}_\mu + \| \mr{e} _n \scr{C} _n p \| ^2 _\mu \\ 
& = & \| p \| ^2 _{H^2 (\mu)} - 2 \nbre \ip{E_n p}{p}_{H^2 (\mu _n )} + \| E_n ^* p \| ^2 _{H^2 (\mu)} \\
& \leq & 2 \| p \| ^2 _{H^2 (\mu )} - 2 \| p \| ^2 _{H^2 (\mu _n )} \\
& = & 2 \int _{\partial \D} | p (\zeta ) | ^2 (1 - g_n (\zeta ) ) \mu (d\zeta )  \rightarrow 0, \ea by the Lebesgue monotone convergence theorem. In conclusion, 
$$ \scr{H} ^+ (\mu ) = \bigvee _{n=1} ^\infty \scr{H} ^+ (\mu _n ), $$ where, here, $\bigvee$ denotes closed linear span.  
\end{proof}

This motivates the following definitions:

\begin{defn}
Let $\mu, \la$ be finite, positive and regular Borel measures on $\partial \D$. We say that $\mu$ is \emph{absolutely continuous with respect to $\la$ in the reproducing kernel sense}, $\mu \ll _{RK} \la$, if the intersection space,
$$ \mr{int} (\mu , \la ) = \scr{H} ^+ (\mu ) \cap \scr{H} ^+ (\la )$$ is norm-dense in $\scr{H} ^+ (\mu )$.  \\

We say that $\mu$ is \emph{reproducing kernel singular} with respect to $\la$, written $\mu \perp _{RK} \la$, if the intesection space is trivial, $\mr{int} (\mu , \la ) = \{ 0 \}$.
\end{defn}

By the previous proposition, $\mu \ll \la$ implies that $\mu \ll _{RK} \la$. The main result of this section will be to show that this new `reproducing kernel' definition of absolute continuity is equivalent to the classical one. 

\begin{lemma} \label{closedembed}
If $\mu \ll _{RK} \la$ then the embedding, $\mr{e} _{\mu, \la} : \mr{int} (\mu ,\la ) \subseteq \scr{H} ^+ (\mu ) \hookrightarrow \scr{H} ^+ (\la )$, is closed with dense domain $\mr{int} (\mu , \la )$. In this case, the co-embedding, $E _{\mu ,\la} : \nbdom E _{\mu ,\la} \subseteq H^2 (\la ) \hookrightarrow H^2 (\mu )$, is densely--defined and closed. Both $\C [ \zeta ]$ and $\scr{K} _\D = \bigvee _{z \in \D } k_z$ are cores for $E_{\mu ,\la }$ and $E_{\mu ,\la } k_z = k_z$, $E_{\mu, \la} p = p$ for all $k_z \in \scr{K} _\D$ and $p \in \C [\zeta ]$. The (closed) self-adjoint and positive semi-definite operator, $T_\mu := E_{\mu ,\la} ^* E_{\mu , \la}$ is $Z^\la-$Toeplitz.
\end{lemma}
\begin{proof}
Let $\mr{e} := \mr{e} _{\mu ,\la}$ and observe that $\mr{e}$ is, trivially, a multiplier by the constant function $1$ from $\scr{H} ^+ (\mu )$ into $\scr{H} ^+ (\la )$. By Proposition \ref{closedmult} and Remark \ref{coeffcore}, $\mr{e}$ is closed on its maximal domain, $\mr{int} (\mu, \la)$, its adjoint acts as 
$$ \mr{e} ^* k^\la _z = k^\mu _z, \quad \mbox{and} \quad \mr{e} ^* k^\la _j = k^\mu _j, $$ on kernels and coefficient kernels and the linear spans of the point evaluation and Taylor coefficient kernels in $\scr{H} ^+ (\la )$ are both cores for $\mr{e} ^*$. Since $E = E_{\mu, \la} := \scr{C} _\mu ^* \mr{e} ^* \scr{C} _\la$, $E$ is closed, the formulas $E p = p$ for $p \in \C [\zeta ]$ and $E k_z = k_z$ are easily verified, and it further follows that $\scr{K} _\D$ and $\C [\zeta ]$ are cores for $E$.

To check that $T:= T_\mu = E_{\mu ,\la } ^* E_{\mu ,\la} \geq 0$ is $\la-$Toeplitz, consider, for any $p, q \in \C [\zeta ]$,
\ba \ip{\sqrt{T} Z^\la p}{\sqrt{T} Z^\la q}_{H^2 (\la )} & = & \ip{ E \zeta \cdot p}{E \zeta \cdot q}_{H^2 (\mu )} \\
& = & \ip{ \zeta p}{\zeta q}_{H^2 (\mu )} = \ip{Z^\mu p}{Z^\mu q}_{H^2 (\mu )} \\
& = & \ip{p}{q}_{H^2 (\mu )} = \ip{Ep}{Eq}_{H^2 (\mu )} \\
& = & \ip{\sqrt{T} p}{\sqrt{T} q}_{H^2 (\la)}. \ea As the polynomials are a core for $E = E_{\mu ,\la}$, this calculation holds on $\nbdom E = \nbdom \sqrt{T}$. Moreover, since $\nbdom E = \nbdom \sqrt{T}$, by polar decomposition of closed operators, and since $\C [\zeta ]$ and $\scr{K} _\D$ are $Z^\la-$invariant cores for $E$, they are also cores for $\sqrt{T}$, and it follows that $\nbdom \sqrt{T} _\mu = \nbdom E$ is also $Z^\la-$invariant.
\end{proof}

\begin{prop} \label{extreme}
Given $\mu, \la \geq 0$, if $\la$ is extreme, then $\mu \ll _{RK} \la$ if and only if $\mu \ll \la$.
\end{prop}
\begin{proof}
Recall that $\la$ is extreme if and only if $H^2 (\la ) = L^2 (\la )$ so that $Z^\la = M^\la _{\zeta}$. In this case, the self-adjoint $\la-$Toeplitz operator $T_\mu \geq 0$ is Toeplitz with respect to the unitary $M^\la _\zeta $. That is, 
$$ \ip{\sqrt{T_\mu} M^\la _\zeta h}{\sqrt{T_\mu} M^\la _\zeta h}_{H^2 (\la )} =  \ip{\sqrt{T_\mu} h}{\sqrt{T_\mu} h}_{H^2 (\la )}, $$ for all $h \in \nbdom \sqrt{T_\mu}$. Hence the quadratic forms for $(M^\la _\zeta) ^* T _\mu M^\la _\zeta$ and $T_\mu$ are the same. By uniqueness of the unbounded Riesz representation, $(M^\la _\zeta) ^* T_\mu M^\la _\zeta = T_\mu$, so that, by Lemma \ref{resolvent},
$$ ( I + (M^\la _\zeta) ^* T_\mu M^\la _\zeta ) ^{-1} = ( I + T _\mu ) ^{-1}, $$ or, equivalently, 
$$ M^\la _\zeta  (I + T_\mu ) ^{-1} = (I + T_\mu ) ^{-1} M^\la _\zeta. $$ This shows that $T_\mu$, and hence $\sqrt{T} _\mu$ are affiliated to the commutant of the unitary operator $M^\la _\zeta$. Since $\C [\zeta ] \subseteq \nbdom \sqrt{T_\mu}$, we conclude that $\sqrt{T _\mu} = M _{\sqrt{T_\mu} 1}$ acts as multiplication by $\sqrt{T _\mu } 1 =: f \in L^2 (\la )$. Since $\sqrt{T _\mu} \geq 0$, we necessarily have that $f\geq 0$, $\la-a.e.$, and we conclude that for any polynomials $p,q$,
$$ \ip{p}{q}_{H^2 (\mu )} = \ip{\sqrt{T_\mu}p}{\sqrt{T_\mu }q}_{H^2 (\la )} = \int _{\partial \D} \ov{p (\zeta )} q(\zeta ) f (\zeta ) ^2 \la (d\zeta ). $$ As in the proof of sufficiency in Theorem \ref{RKdominate}, we conclude that the above formula holds for any $g,h \in \C [\zeta ] + \ov{\C [\zeta ]}$, which is dense in $\scr{C} (\partial \D )$ and $L^\infty (\mu)$. In particular, the formula holds for all simple functions and characteristic functions of Borel sets. Since $f \in L^2 (\la )$, $f^2 \in L^1 (\la )$ and it follows that 
$$ f^2 = \frac{\mu (d\zeta )}{\la (d\zeta )}, $$ is the Radon--Nikodym derivative of $\mu$ with respect to $\la$.
\end{proof}

To prove that absolute continuity in the reproducing kernel sense is equivalent to absolute continuity in general, we will appeal to B. Simon's Lebesgue decomposition theory for positive quadratic forms in Hilbert space \cite{Simon-forms}, \cite[Supplement to VIII.7]{RnS}. Let $\hat{\cH}  (\fq)$ be the Hilbert space completion of $\nbdom \fq$ with respect to the inner product $\ip{\cdot}{\cdot} _\cH + \fq (\cdot , \cdot )$, and let $\mr{j} _{\fq} : \nbdom \fq \hookrightarrow \hat{\cH} (\fq )$ denote the formal embedding. Further define the co-embedding $E_\fq : \hat{\cH} (\fq )\hookrightarrow \cH$ by 
$$ E _\fq (\mr{j} _\fq (x) ) := x, \quad \quad x \in \nbdom \fq. $$ By construction, $\mr{j} _\fq$ is densely--defined, has dense range, and $E_\fq$ is contractive with dense range in $\cH$. Hence $E_\fq$ extends by continuity to a contraction, also denoted by $E_\fq$, $E_\fq : \hat{\cH} (\fq ) \hookrightarrow \cH$. 

\begin{lemma}
A densely--defined and positive semi-definite quadratic form, $\fq$, in $\cH$, is closeable if and only if $\mr{j} _\fq$ is closeable, or equivalently, if and only if $E_\fq$ is injective.
\end{lemma}
This lemma is a straightforward consequence of the definitions, see also \cite{Simon-forms}.

\begin{thm}[Simon--Lebesgue decomposition of positive forms] \label{formLD}
Let $\fq \geq 0$ be a positive semi-definite quadratic form with dense form domain, $\nbdom \fq$, in a separable, complex Hilbert space, $\cH$. Then $\fq$ has a unique Lebesgue decomposition, $\fq = \fq_{ac} + \fq _s$, where $0 \leq \fq_{ac}, \fq _s \leq \fq$ in the quadratic form sense, $\fq _{ac}$ is the maximal absolutely continuous form less than or equal to $\fq$ and $\fq _s$ is a singular form.  \\

\noindent If $P_s$ denotes the projection onto $\nbker E_\fq$, and $P_{ac} = I - P_s$, then $\fq_{ac}$ is given by the formula,
\be \fq_{ac} (x,y) = \ip{ \mr{j} _\fq (x) }{(P_{ac} - E _\fq ^* E_\fq ) \mr{j} _\fq (y)}_{\hat{\cH} (\fq)} = \ip{ \mr{j} _\fq (x) }{P_{ac} \, \mr{j} _\fq (y)}_{\hat{\cH} (\fq)}  - \ip{x}{y}_{\cH}. \label{SimonRNform} \ee
\end{thm}
In the above theorem statement, recall that we defined the notions of an absolutely continuous or singular positive quadratic form in the introduction. Namely, a positive semi-definite and densely--defined quadratic form, $\fq : \nbdom \fq \times \nbdom \fq \rightarrow \cH$, $\nbdom \fq \subseteq \cH$, is called absolutely continuous if it is closeable, and singular if the only absolutely continuous and positive semi-definite form it dominates is the identically zero form.

\begin{remark} \label{measform}
If, now, $\mu, \la \geq 0$ are measures on the circle, we can take $\cH := L^2 (\la )$ or $H^2 (\la)$, and define $\fq _\mu \geq 0$ on a dense form domain in $\cH$ by the formula 
$$ \fq_\mu (f,g) = \int _{\partial \D} \ov{f(\zeta)} g(\zeta ) \mu (d\zeta ). $$ For example, if $\cH = L^2 (\la )$, one can take $\nbdom \fq_\mu = \scr{C} (\partial \D )$, the continuous functions. In this case, by the remark on \cite[p. 381]{Simon-forms}, the quadratic form Lebesgue decomposition of $\fq_\mu$ coincides with the classical Lebesgue decomposition of $\mu$ with respect to $\la$. Namely, in this case, the absolutely continuous part of $\fq_\mu$, $\fq_{\mu ; ac}$ is equal to $\fq_{\mu _{ac}}$, the positive form of the absolutely continuous part of $\mu$ with respect to $\la$, $\mu _{ac}$, and $\fq_{\mu ;s} = \fq _{\mu _s}$. In particular, $\fq_T := \ov{\fq_{ac}}$ is the form of the positive semi-definite, self-adjoint operator $T = M_f \geq 0$, where $f \in L^1 (\la )$ is the Radon--Nikodym derivative of $\mu$ with respect to $\la$. This follows because, as observed by Simon, in this case his construction of the absolutely continuous and singular parts of $\fq_\mu$ essentially reduces to von Neumann's functional analytic proof of the Lebesgue decomposition and Radon--Nikodym theorem in \cite[Lemma 3.2.3]{vN3}. See also \cite[Section 5]{HdeSnoo-forms}, which arrives at the same conclusion with the choice of form domain, $\nbdom \fq_\mu \subseteq L^2 (\la )$, equal to the simple functions, \emph{i.e.} linear combinations of characteristic functions of Borel sets. 
\end{remark}

\begin{thm} \label{RNformula}
Let $\fq \geq 0$ be a densely--defined and positive semi-definite quadratic form in a separable complex Hilbert space, $\cH$. If $\fq_T = \ov{\fq_{ac}}$ is the closure of $\fq_{ac}$, then $(I+T) ^{-1} = E _\fq E_\fq ^*$, where $E_\fq : \hat{\cH} (\fq ) \hookrightarrow \cH$ is the contractive co-embedding.
\end{thm}

\begin{lemma} \label{closeableform}
Let $A : \nbdom A \subseteq \cH \rightarrow \cH$ be a densely-defined linear operator. Then $A$ is closeable if and only if the positive semi-definite quadratic form $\fq _{A^*A} (x,y) := \ip{Ax}{Ay}_{\cH}$, with form domain $\nbdom \fq _{A^*A} := \nbdom A$, is closeable.
\end{lemma}
In the above statement, note that $A^*A$ is not defined if $A$ is not closeable.
\iffalse
\begin{lemma} \label{adjoints}
Let $A \in \scr{L} (\cH )$ be bounded and let $B: \nbdom B \subseteq \cH \rightarrow \cH$ be closeable and densely-defined. If $C:= AB : \nbdom B \rightarrow \cH$ is also closeable, then $A^* : \nbdom C^* \rightarrow \nbdom B^*$ and $C^* = B^* A^*$.
\end{lemma}
\begin{proof}
Given any $x \in \nbdom B = \nbdom C$ and a fixed $y \in \nbdom C^*$,
$$ \ip{x}{C^* y}_{\cH}  =  \ip{ABx}{y}_{\cH} = \ip{Bx}{A^*y}_{\cH}.$$
By definition of the Hilbert space adjoint, it follows that $A^* y \in \nbdom B^*$ and that $B^*A^*y=C^*y$.
\end{proof}
\fi

\begin{proof}[Proof of Theorem \ref{RNformula}]
Let $(x _j )_{j=1} ^\infty \subseteq \nbdom \fq$ be a sequence with dense linear span. Apply Gram-Schimdt orthogonalization to $(x_j)$ with respect to the $\fq +\mf{id}-$inner product of $\hat{\cH} (\fq )$. This yields a countable basis $(y_j )_{j=1} ^\infty \subseteq \nbdom \fq$, so that the sequence $(\mr{j} _\fq (y_j) )$ is an orthonormal basis of $\hat{\cH} (\fq )$. Hence,
$$ E _\fq E _\fq ^* = E _\fq I_{\hat{\cH} _\fq} E_\fq ^* = \sum _{j=1} ^\infty \ip{E_\fq \mr{j} _\fq (y_j)}{\cdot}_{\cH} E_\fq \mr{j}_\fq (y_j ) = \sum \ip{y_j}{\cdot}_\cH y_j. $$
By \cite[Theorem 2.1, Corollary 2.3]{Simon-forms}, see Theorem \ref{formLD} and Equation (\ref{SimonRNform}) above,
$$ \fq_{ac} (x,y) + \ip{x}{y}_{\cH} = \fq_{I+T} (x,y) = \ip{\sqrt{I+T}x}{\sqrt{I+T}y}_{\cH} = \ip{\mr{j} _\fq (x)}{P_{ac} \mr{j} _\fq (y)}_{\hat{\cH} (\fq )}, $$ for any $x,y \in \nbdom \fq \subseteq \nbdom \fq _{ac} \subseteq \nbdom \sqrt{I+T}$. Hence, for any $x,y \in \nbdom T$, 
\ba q_{E_\fq E_\fq ^*} \left( (I+T ) x , (I+T) y \right) & = & \ip{ \sqrt{E_\fq E_\fq ^*} (I+T) x}{\sqrt{E_\fq E_\fq ^*} (I+T) y}_{\cH} \\
& = & \sum _{j=1} ^\infty \ip{(I+T)x}{y_j}_\cH \ip{y_j}{(I+T) y}_{\cH } \\
& = & \sum \ip{\sqrt{I+T} x}{\sqrt{I+T} y_j}_\cH \ip{\sqrt{I+T} y_j}{\sqrt{I+T} y}_{\cH }  \\
& = & \sum  \ip{P_{ac} \, \mr{j} _\fq (x)}{\mr{j}_\fq (y_j)}_{\fq+\mf{id}} \ip{\mr{j} _\fq (y_j)}{P_{ac} \, \mr{j}_\fq (y)}_{\fq+\mf{id}} \\
& = & \ip{P _{ac} \, \mr{j} _\fq (x)}{P_{ac} \, \mr{j}_\fq (y)}_{q+\mf{id}} = q_{I+T} (x,y ) \\
& = & \ip{x}{(I+T) y}_{\cH}. \ea 
That is, the (closeable) quadratic forms of $I+T$ and $(I+T) E_\fq E_\fq ^* (I+T)$ agree on $\nbdom T$, which is a core for $\sqrt{I+T}$, and a form-core for $\fq_{I+T}$. Moreover, $\scr{D} := \nbran \sqrt{I+T} \cap \nbdom \sqrt{I+T}$ is a core for $\sqrt{I+T}$, so that for all $x = \sqrt{I+T} y \in \scr{D}$,
\ba \ip{x}{x} _{\cH} & = & \ip{\sqrt{I+T} y}{\sqrt{I+T} y}_{\cH} = \ip{(I+T) y}{E_\fq E_\fq ^* (I+T) y}_{\cH} \\
& = & \ip{ \sqrt{I+T} x}{E_\fq E_\fq ^* \sqrt{I+T}x}_{\cH}. \ea
That is, the (bounded) positive quadratic form of the identity, $I$, agrees with the quadratic form of $\sqrt{I+T} E_\fq E_\fq ^* \sqrt{I+T}$ on the dense subspace $\nbdom \sqrt{I+T}$. Here, if $V:= E _\fq ^* \sqrt{I+T}$, this is a closeable operator by Lemma \ref{closeableform}. In fact, $V$ extends by continuity to an isometry, since $\fq _{V^*V} = \fq _I | _{\nbdom \sqrt{I+T}}$. Moreover, $E_\fq$, and hence $\sqrt{E_\fq E_\fq ^*}$ have dense range, so that $\sqrt{E_\fq E_\fq ^*} \sqrt{I+T}$ extends to an isometry with dense range, \emph{i.e.} a unitary. For any $x,y \in \nbdom \sqrt{I+T}$, we have that 
$$ \ip{x}{y}_\cH  = \ip{\sqrt{I+T} x}{E_\fq E_\fq ^* \sqrt{I+T} y}_\cH. $$ Hence, by definition of the adjoint, for any $y \in \nbdom \sqrt{I+T}$, $E _\fq E_\fq ^* \sqrt{I+T}y \in \nbdom \sqrt{I+T}$, and 
$$ \sqrt{I+T} E_\fq E_\fq ^* \sqrt{I+T} y = y. $$  Hence for any $x = \sqrt{I+T} ^{-1} g$ and $y = \sqrt{I+T} ^{-1} h$,
\ba \ip{g}{(I+T) ^{-1} h}_{\cH} & = & \ip{x}{y}_{\cH} = \ip{E_\fq ^* \sqrt{I+T}x}{E_\fq ^* \sqrt{I+T}y}_{\cH} \\
& = & \ip{g}{E_\fq E_\fq ^* h}_{\cH}, \ea and we conclude, by the Riesz lemma for bounded sesquilinear forms, that 
$$ (I+T) ^{-1} = E_\fq E_\fq ^*. $$
\end{proof}

%%%%%%%%%%%% Rob: Double-check the end of the previous proof.

\begin{thm} \label{RKac}
Let $\mu, \la$ be positive, finite and regular Borel measures on $\partial \D$. Then $\mu \ll \la$ if and only if $\mu \ll _{RK} \la$. 

In this case, the co-embedding, $E _{\mu, \la} : \nbdom E _{\mu, \la} \subseteq H^2 (\la ) \hookrightarrow H^2 (\mu )$ is closed, and its domain is $Z^\la-$invariant. Both $\C [\zeta ]$ and $\scr{K} _\D = \bigvee k_z$ are cores for $E _{\mu , \la}$, $E_{\mu, \la} p = p$ and $E_{\mu ,\la} k_z = k_z$ for any $p \in \C [\zeta ]$ or Szeg\"o kernel $k_z$. The self-adjoint and positive semi-definite operator, $T_\mu := E _{\mu , \la} ^* E _{\mu , \la}$, is $\la-$Toeplitz and $$T _\mu = P_{H^2 (\la ) } M_{f} |_{H^2 (\la )}; \quad \quad f := \frac{d\mu}{d\la}, $$ in the quadratic form sense, \emph{i.e.}
$$ \ip{\sqrt{T} _\mu p}{\sqrt{T} _\mu q}_{H^2 (\la )} = \ip{M_{\sqrt{f}} \, p}{M_{\sqrt{f}} \,  q}_{L^2 (\la )}; \quad \quad \forall \ p, q \in \C [\zeta]. $$ 
\end{thm}

\begin{proof}
Necessity was established in Proposition \ref{necessary} and sufficiency, in the case where $\la$ is extreme, was proven in Proposition \ref{extreme}. To prove sufficiency in general, assume that $\mu \ll _{RK} \la$ so that the intersection space $\mr{int} (\mu , \la)$ is dense in the space of $\mu-$Cauchy transforms, $\scr{H} ^+ (\mu )$. Note that $\mu \ll _{RK} \la$ if and only if $\mu + \la  \ll _{RK} \la$ (and the same is true for $\ll$). This follows from Aronszajn's `sums of kernels' theorem as stated in Subsection \ref{ss-rkhs}. By Lemma \ref{closedembed}, $\mr{e} _{\mu +\la, \la} : \mr{int} (\mu + \la, \la ) \subseteq \scr{H} ^+ (\mu +\la ) \hookrightarrow \scr{H} ^+ (\la )$ is a closed embedding, $E_{\mu +\la, \la } := \scr{C} _{\mu +\la} ^* \mr{e} _{\mu +\la, \la} ^* \scr{C} _\la$, is a closed co-embedding with $\scr{K} _\D = \bigvee k_z$ and $\C [\zeta ]$ as cores, and $T := T_{\mu +\la} = E_{\mu +\la, \la } ^* E_{\mu +\la, \la} \geq 0$ is self-adjoint, positive semi-definite, densely--defined and $\la-$Toeplitz. 

By Remark \ref{measform} and Theorem \ref{RNformula} above if $\hat{E} : L^2 (\mu + \la ) \hookrightarrow L^2 (\la )$ is the contractive co-embedding, and $\hat{T} := M_f$, where $f \geq 0$ $\la-a.e.$, $f \in L^1 (\la )$ is the self-adjoint, positive semi-definite multiplication operator by the Radon--Nikodym derivative, $f = \frac{d\mu}{d\la}$, then $(I +\hat{T}) ^{-1} = \hat{E} \hat{E}^*$. On the other hand, $E:= E_{\mu + \la , \la} : \nbdom E_{\mu +\la , \la} \subseteq H^2 (\la ) \hookrightarrow H^2 (\mu +\la )$ is closed and bounded below by $1$, and so has a contractive inverse, namely, for any $h \in \scr{K} _\D$ or in $\C [\zeta ]$, 
$$ \hat{E} E h = h \in H^2 (\la ), \quad \mbox{and,} \quad E \hat{E} h = h \in H^2 (\mu +\la ). $$ Hence, 
$$ \hat{E} E | _{\nbdom E} = I_{H^2 (\la)} |_{\nbdom E}, \quad \mbox{and} \quad E \hat{E} | _{H^2 (\mu + \la ) } = I _{\mu +\la}. $$ Observe that the contractive co-embedding, $\hat{E}$ is necessarily injective and has dense range, so that it has a closed, potentially unbounded inverse, $\hat{E} ^{-1}$, which is densely--defined. Since $(I +\hat{T} ) ^{-1} = \hat{E} \hat{E} ^*$, we conclude that $I +\hat{T} = \hat{E} ^{-*} \hat{E} ^{-1}$. Hence, for any $p, q \in \C [\zeta] \subseteq H^2 (\la )$, 
\ba \ip{\sqrt{I +\hat{T}} p}{\sqrt{I+ \hat{T}} q}_{L^2 (\la )} & = & \ip{\hat{E} ^{-1} p}{\hat{E} ^{-1} q}_{L^2 (\mu + \la )} \\
& = & \ip{p}{q}_{H^2 (\mu +\la )} = \ip{Ep}{Eq}_{H^2 ( \mu +\la  )} \\
& = & \ip{\sqrt{T} p}{\sqrt{T} q}_{H^2 (\la )}. \ea 
This calculation shows that the `compression' of $I + \hat{T}$ to the intersection of its domain with the subspace $H^2 (\la)$ is equal to $T$, in this quadratic form sense. In particular $T - I \geq 0$ is the compression of $\hat{T} = M_f$ to $H^2 (\la )$, where $f \in L^1 (\la )$ is the Radon--Nikodym derivative of $\mu$ with respect to $\la$. In conclusion, for any polynomials $p,q$,
\ba \int _{\partial \D} \ov{p(\zeta)} q (\zeta) \mu (d\zeta ) & = & \ip{p}{q}_{H^2 (\mu )}  \\
& = & \ip{\sqrt{T}p}{\sqrt{T}q}_{H^2 (\la)} - \ip{p}{q}_{H^2 (\la)} \\
& = & \ip{\sqrt{\hat{T}} p}{\sqrt{\hat{T}}q}_{L^2 (\la )} \\
& = & \ip{M_{\sqrt{f}} p}{M_{\sqrt{f}} q}_{L^2 (\la)} \\
& = & \int _{\partial \D} \ov{p(\zeta)} q(\zeta) f (\zeta ) \la (d\zeta ). \ea 
As in the second proof of sufficiency of Theorem \ref{RKdominate}, this equality can be extended to arbitrary $g,h \in \C [\zeta ] + \ov{\C [\zeta ] }$, so that by \mbox{Weierstra\ss} approximation, $\mu \ll \la$ with Radon--Nikodym derivative $f  \geq 0$, $f \in L^1 (\la )$. 
\end{proof}

\section{Lebesgue decomposition via reproducing kernels}

By Theorem \ref{RKac}, our definition of reproducing kernel absolute continuity is equivalent to the classical definition of absolute continuity for finite, positive and regular Borel measures on the complex unit circle. In particular, if $\mu \ll \la$, it follows that the intersection space of $\mu$
and $\la-$Cauchy transforms is dense in the space of $\mu-$Cauchy transforms. Hence, if $\mu = \mu _{ac} + \mu _s$ is the Lebesgue decomposition of $\mu$ with respect to $\la$, then $\mr{int} (\mu _{ac} , \la )$ is dense in $\scr{H} ^+ (\mu _{ac} )$, and since $\mu \geq \mu _{ac}$, $\mr{int} (\mu _{ac} , \la ) \subseteq \mr{int} (\mu , \la )$. That is, if $\mu _{ac} \neq 0$, it follows that $\mr{int} (\mu, \la ) \neq \{ 0 \}$ is not trivial. This raises several natural questions: How can we identify the space of $\mu _{ac}-$Cauchy transforms? Is $\mr{int} (\mu , \la ) ^{-\mu } := \mr{int} (\mu , \la ) ^{ -\| \cdot \| _{\mu} }$ equal to the space of $\mu _{ac}-$Cauchy transforms? We will see that the answer to the second question is positive if $\la$ is non-extreme, but that in general, $\mr{int} (\mu , \la ) ^{-\mu}$ is not the space of Cauchy transforms of any positive measure, see Corollary \ref{LDvsformLD} and Example \ref{Lebeg2}.

\begin{thm} \label{CTsubspace}
If $\scr{M}$ is a RKHS in $\D$ that embeds contractively in $\scr{H} ^+ (\mu )$, then $\scr{M} = \scr{H} ^+ (\ga )$ for a positive measure, $\ga$, $\ga \leq \mu$, if and only if $\mr{e} : \scr{M} \hookrightarrow \scr{H} ^+ (\mu )$ is such that the positive semi-definite contraction $\tau := \mr{e} \mr{e} ^*$ is $V_\mu-$Toeplitz.

In this case, $\scr{M} = \scr{H} ^+ (\ga ) = \scr{R} (\mr{e} )$, and the complementary space of $\scr{H} ^+ (\ga)$ in $\scr{H} ^+ (\mu )$ is $\scr{H} ^+ (\nu )$, for a positive measure, $\nu$, where $k^\nu = k^\mu - k^\ga$ so that $\mu = \ga + \nu$.
\end{thm}
\begin{proof}
First, if $\scr{H} ^+ (\ga) = \scr{M} \subseteq \scr{H} ^+ (\mu )$ is contractively contained, then, $\mr{e} : \scr{H} ^+ (\ga) \hookrightarrow \scr{H} ^+ (\mu )$ is trivially a (contractive) multiplier so that, as before, $\mr{e} ^* k_z ^\mu = k_z ^\ga$, and
$$ \mr{e} ^* V^\mu = V^\ga \mr{e} ^*. $$ In conclusion, 
$$ V ^{\mu *} \mr{e} \mr{e} ^* V^\mu = \mr{e} V ^{\ga *} V^\ga \mr{e} ^* = \mr{e} \mr{e} ^* = \tau. $$

Conversely, if $\tau = \mr{e} \mr{e} ^*$ is $V^\mu-$Toeplitz and contractive, then, as in the proof of Theorem \ref{RKdominate}, $T := \scr{C} _\mu ^* \mr{e} \mr{e} ^* \scr{C} _\mu$ is a contractive $Z ^\mu-$Toeplitz operator and we can appeal to the Riesz--Markov theorem to show that there is a $\ga \geq 0$, so that $T = P_{H^2 (\mu)} M_f | _{H^2 (\mu )}$, where $f \geq 0$, $\| f \| _\infty \leq 1$ is the Radon--Nikodym derivative of $\ga$ with respect to $\mu$. Namely, one can define a linear functional, $\hat{\mu} _T$, on $\C [\zeta ] + \ov{\C [\zeta]} \subseteq \scr{C} (\partial \D )$, by
$$ \hat{\mu} _T (p + \ov{q} ) := \ip{1}{Tp}_{H^2 (\mu )} + \ip{q}{T1}_{H^2 (\mu )}. $$
It is easy to check that $\hat{\mu} _T$ is bounded and positive using the Fej\'er--Riesz theorem, as in the proof of Theorem \ref{RKdominate}. The fact that $T$ is a positive semi-definite $Z^\mu-$Toeplitz contraction ensures that $\hat{\mu} _T$ extends to a bounded, positive linear functional on $\scr{C} (\partial \D )$, and that $\hat{\mu} _T \leq \hat{\mu}$, so that $\hat{\mu} _T = \hat{\ga}$ for some finite, regular and positive Borel measure, $\ga \leq \mu$, by the Riesz--Markov theorem. 

By  Theorem \ref{compRKHS}, the complementary space, $\scr{R} ^c (\mr{e} )$, of $\scr{R} (\mr{e}) = \scr{H} ^+ (\ga)$ is a RKHS in $\D$ with reproducing kernel $k ' (z,w) = k^\mu (z,w) - k^\ga (z,w)$ and it is contractively contained in $\scr{H} ^+ (\mu )$, by the inclusion theorem. Moreover, if $\mr{j} : \scr{R} ^c (\mr{e} ) \hookrightarrow \scr{H} ^+ (\mu )$ is the contractive embedding, then it follows that $\mr{j} \mr{j} ^* = I - \mr{e} \mr{e} ^* \geq 0$ is also a positive semi-definite $V^\mu-$Toeplitz contraction. Hence, by the first part of the proof, $\scr{H} = \scr{H} ^+ (\nu)$ for a positive measure, $\nu$. Finally, since $k^\mu = k^\ga + k^\nu$, we obtain that $\mu = \ga + \nu$.
\end{proof}

\begin{lemma} \label{coinvint}
Given any $\mu, \la$, the intersection space $\mr{int} (\mu ,\la)$, is both $V_\la$ and $V_\mu-$co-invariant, and $V_\la ^* | _{\mr{int} (\mu, \la )} = V_\mu ^* |_{\mr{int} (\mu , \la )}$. 
\end{lemma}
\begin{proof}
This is immediate, by Lemma \ref{bwshift}, since both $V_\mu ^*$ and $V_\la ^*$ act as `backward shifts' on power series. 
\end{proof}

\begin{lemma} \label{invint}
If $\la$ is non-extreme, then the intersection space, $\mr{int} (\mu, \la)$, is $V_\mu-$reducing. If $\la$ is extreme, then $\mr{int} (\mu , \la )$ is $V_\mu-$reducing (and $V_\la-$reducing) if and only if $V_{\mu} | _{\mr{int} (\mu , \la)} = V_\la | _{\mr{int} (\mu, \la )}$. 
\end{lemma}
\begin{proof}
By Lemma \ref{bwshift}, if $h \in \mr{int} (\mu, \la)$, then $V_\mu h \in \scr{H} ^+ (\mu )$ and 
\be \label{same} (V_\mu h) (z) = z h(z) + (V_\mu h) (0) 1= (V_\la h) (z) - (V_\la h) (0) 1 + (V_\mu h) (0) 1. \ee Hence, if $c:= (V_\la h) (0) - (V_\mu h) (0) \in \C$ and $\la$ is non-extreme, then both $V_\la h$ and $c1$ belong to $\scr{H} ^+ (\la )$ so that $V_\mu h \in \scr{H} ^+ (\la ) \cap \scr{H} ^+ (\mu ) = \mr{int} (\mu, \la)$. Recall that $\la$ is extreme if and only if $\scr{H} ^+ (\la )$ does not contain the constant functions. Hence, if $\la$ is extreme then $\mr{int} (\mu ,\la)$ will be $V_\mu-$reducing if and only if $(V_\mu h) (0) = (V_\la h) (0)$ for all $h \in \mr{int} (\mu , \la )$. By Equation (\ref{same}), this happens if and only if $V_\mu h = V_\la h$.  
\end{proof}

\begin{cor} \label{reducemeas}
If $\scr{M} \subseteq \scr{H} ^+ (\mu )$ is a $V_\mu-$reducing subspace, then $\scr{M} = \scr{H} ^+ (\ga )$ for some $\ga \leq \mu$. Moreover, $\scr{M} ^\perp = \scr{H} ^+ (\ga ')$ for some $\mu \geq \ga' \geq 0$ so that $\ga + \ga ' = \mu$.
\end{cor}
\begin{proof}
Let $P$ be the orthogonal projection of $\scr{H} ^+ (\mu )$ onto $\scr{M}$. Then if $\mr{e} : \scr{M} \hookrightarrow \scr{H} ^+ (\mu )$ is the isometric embedding, $P= \mr{e} \mr{e} ^*$. Hence, $$ V_\mu ^* \mr{e} \mr{e} ^* V_\mu = V_\mu ^* P V_\mu = V_\mu ^* V_\mu P = P = \mr{e}\mr{e} ^*,$$ so that $\tau = \mr{e} \mr{e} ^*$ is $V ^\mu-$Toeplitz and $\scr{M} = \scr{H} ^+ (\ga )$ for some $0 \leq \ga \leq \mu$, and $\scr{M} ^\perp = \scr{H} ^+ (\ga ' )$, by Theorem \ref{CTsubspace}.
\end{proof}

\begin{thm} \label{LDviaRKne}
Let $\mu, \la \geq 0$ be finite, positive and regular Borel measures on $\partial \D$. If the intersection space, $\mr{int} (\mu, \la)$, is $V ^\mu-$reducing and $\mu = \mu _{ac} + \mu _s$ is the Lebesgue decomposition of $\mu$ with respect to $\la$, then 
$$ \scr{H} ^+ (\mu ) = \scr{H} ^+ ( \mu _{ac} ) \oplus \scr{H} ^+ (\mu _s ).$$ In this case,
$$ \scr{H} ^+ (\mu _{ac} ) = \mr{int} (\mu , \la ) ^{-\mu}, \quad \mbox{and} \quad \scr{H} ^+ (\mu _s ) \cap \scr{H} ^+ (\la ) = \{ 0 \}. $$ That is, $\mu_{ac}$ is the largest positive measure $\leq \mu$ which is RK-ac with respect to $\la$, and $\mu _s$ is RK-singular with respect to $\la$.  
\end{thm}
In particular, $\mr{int} (\mu, \la)$ will be $V^\mu-$reducing if $\la$ is non-extreme by Lemma \ref{invint}.
\begin{proof}
By Theorem \ref{RKac}, we have that $\mr{int} (\mu _{ac}, \la) \subseteq \mr{int} (\mu, \la )$ is dense in $\scr{H} ^+ (\mu _{ac})$. Since we assume that $\mr{int} (\mu , \la)$ is $V_\mu-$reducing, its closure, $\mr{int} (\mu , \la ) ^{-\mu}$, is also $V_\mu-$reducing and then $\mr{int} (\mu , \la ) ^{-\mu} = \scr{H} ^+ (\ga)$ for some $0 \leq \ga \leq \mu$ by Corollary \ref{reducemeas}. By construction $\ga \ll _{RK} \la$ so that $\ga \ll \la$ by Theorem \ref{RKac}. By maximality, $\ga \leq \mu _{ac}$ and by construction $\mr{int} (\mu , \la ) \subseteq \mr{int} (\ga , \la )$. However, we also have that $\mr{int} (\mu _{ac} , \la ) \subseteq \mr{int} (\mu , \la ) \subseteq \mr{int} (\ga , \la )$. Hence if $\mr{e} : \scr{H} ^+ (\mu _{ac} ) \hookrightarrow \scr{H} ^+ (\mu )$ is the contractive embedding, then $\mr{e}$, restricted to the dense subspace $\mr{int} (\mu _{ac} , \la ) \subseteq \scr{H} ^+ (\mu _{ac} )$ defines a contraction into $\scr{H} ^+ (\ga )$, which is isometrically contained in $\scr{H} ^+ (\mu )$. It follows that $\mr{e}$ extends by continuity to a contractive embedding of $\scr{H} ^+ (\mu _{ac} )$ into $\scr{H} ^+ (\ga)$. Hence, by Theorem \ref{RKdominate}, $\mu _{ac} \leq \ga$ and we conclude that $\mu _{ac} =\ga$.
\end{proof}

\begin{cor}
Let $\mu, \la \geq 0$ be finite, positive and regular Borel measures on $\partial \D$. If $\mr{int} (\mu ,\la) = \{ 0 \}$ so that $\mu \perp _{RK} \la$, then $\mu \perp \la$. If $\mr{int} (\mu , \la )$ is either $V_\mu$ or $V_\la-$reducing then $\mu \perp \la$ if and only if $\mu \perp _{RK} \la$. In particular, if either $\mu$ or $\la$ is non-extreme then $\mu \perp \la$ if and only if $\mu \perp _{RK} \la$.
\end{cor}

While the previous two results give quite a satisfactory description of the Lebesgue decomposition of $\mu$ with respect to $\la$ in terms of reproducing kernel theory in the case where the intersection space, $\mr{int} (\mu ,\la )$ is $V^\mu-$reducing, this is not generally the case, as the next proposition and example show. 

In the proposition statement below, recall the definition of the lattice operations $\vee, \wedge$ on positive kernels and the definition of the isometries $U_\vee : \cH (K +k ) \rightarrow \cH (K) \oplus \cH (k)$ and $U_\wedge : \cH (K \wedge k) = \cH (K) \cap \cH (k) \rightarrow \cH (K) \oplus \cH (k)$, see Subsection \ref{ss-rkhs}. 

\begin{prop} \label{notred}
If $\mu + \la$ is extreme, then $\mr{int} (\mu, \la)$ is $V_\mu-$reducing. If $\mu, \la$ are both extreme but $\mu + \la$ is non-extreme, then $\mr{int} (\mu , \la)$ is non-trivial, but not $V_\mu-$reducing and 
$$ \mr{int} (\mu , \la ) \supseteq U_\wedge ^* \bigvee _{j=1} ^\infty  (V_\mu ^* \oplus V_\la ^* ) ^j U_\vee 1. $$ 
\end{prop}

\begin{lemma}
Let $\mu, \la \geq 0$ be positive, finite and regular Borel measures on $\partial \D$. Consider the reproducing kernel Hilbert spaces of $\mu, \la$ and $\mu + \la-$Cauchy transforms in $\partial \D$, $\scr{H} ^+ (\mu ) = \cH (k^\mu )$, $\scr{H} ^+ (\la ) = \cH (k ^\la )$ and $\scr{H} ^+ (\mu + \la ) = \cH ( k^\mu + k^\la )$. Then, $$ U_\vee V_{\mu + \la } = V_\mu \oplus V_\la U_\vee $$ and $\nbran U_\vee$ is $V_\mu \oplus V_\la-$invariant so that $\nbran U_\wedge$ is $V_\mu \oplus V_\la$ co-invariant, and $$U_\wedge V^{\mu *} | _{\mr{int} (\mu, \la )} = U_\wedge V^{\la *} |_{\mr{int} (\mu, \la )} = (V^\mu \oplus V^\la ) ^* U_\wedge. $$ Moreover, we have that $V_\mu | _{\mr{int} (\mu, \la)} = V_\la | _{\mr{int} (\mu, \la )}$ so that $\mr{int} (\mu, \la) = \cH (k^\mu \wedge k^\la)$ is both $V_\mu$ and $V_\la-$invariant if and only if $\nbran U_\wedge$ is $V_\mu \oplus V_\la-$invariant.
\end{lemma}
\begin{proof}
The intertwining formulas are easily verified. The range of $U _\wedge : \scr{H} ^+ (\mu + \la ) \rightarrow \scr{H} ^+ (\mu ) \oplus \scr{H} ^+ (\la)$ is $V_\mu \oplus V_\la-$reducing if and only if, for any $h \oplus -h \in \nbran U _\wedge$, $h \in \mr{int} (\mu, \la )$,
$$ V _\mu h \oplus - V_\la h = g \oplus -g, $$ for some $g \in \mr{int} (\mu, \la)$. Clearly this happens if and only if $V_\mu | _{\mr{int} (\mu, \la)} = V_\la | _{\mr{int} (\mu, \la )}$.
\end{proof}

\begin{proof}[Proof of Proposition \ref{notred}.]
If $\la + \mu$ is extreme then $\scr{H} ^+ (\mu + \la) = \scr{H} ^+ (\mu ) + \scr{H} ^+ (\la )$ does not contain the constant functions. Hence both $\la$ and $\mu$ must also be extreme. In this case $V_\mu, V_\la$ and $V_{\mu + \la}$ are all unitary operators. We know that $\nbran U_\vee$ is always $V_\mu \oplus V_\la-$invariant. On the other hand since $\mu + \la$ is extreme, $V_{\mu + \la}$ is unitary, hence surjective, and 
$$ \scr{H} ^+ ( \mu + \la ) = \bigvee (k_z ^{\mu +\la} - k_0 ^{\mu + \la} ), $$ so that 
$$ \nbran U_\vee = \bigvee (k_z ^{\mu} - k_0 ^{\mu} ) \oplus (k_z ^\la - k_0 ^\la). $$ Hence, 
$$ (V_\mu ^* \oplus V_\la ^*) \nbran U_\vee = \bigvee k_z ^\mu \ov{z} \oplus k_z ^\la \ov{z} \subseteq \nbran U_\vee. $$ It follows that $\nbran U_\vee$ is $V_\mu \oplus V_\la-$reducing, so that $\nbran U_\wedge = \nbran U_\vee ^\perp$ is also reducing. The previous lemma now implies that $\mr{int} (\mu, \la)$ is $V_\mu-$reducing. 

If, on the other hand, $\mu , \la$ are both extreme but $\mu + \la$ is not, then $V^\mu, V^\la$ are both unitary but $V^{\mu +\la}$ is not. Hence, 
since $1 \perp \nbran V^{\mu + \la}$, $1 \in \scr{H} ^+ (\mu + \la)$, we have that 
$$ U_\vee 1 \perp V ^\mu \oplus V^\la \nbran U_\vee, $$ or, equivalently, 
$$ (V ^\mu \oplus V^\la ) ^* U_\vee 1 \perp \nbran U_\vee. $$ Since $V^\mu \oplus V^\la$ is unitary, it follows that 
$$ 0 \neq (V ^\mu \oplus V^\la ) ^* U_\vee 1 \in \nbran U _\wedge, $$ so that $\mr{int} (\mu, \la ) \neq \{ 0 \}$. Since $\nbran U_\vee$ is not $V ^\mu \oplus V^\la-$reducing, neither is $\nbran U_\wedge$, and hence $\mr{int} (\mu, \la)$ is not $V_\mu-$reducing by the previous lemma.
\end{proof}

\begin{eg}[Lebesgue measure on the half circles] \label{LebesgueEg}
Let $m_{\pm}$ be normalized Lebesgue measure restricted to the upper and lower half-circles. Then $m = m_+ + m_-$, and $m_+ \perp m_-$. Note that both $m_\pm$ are extreme since $\frac{dm_\pm}{dm} = \chi _{\partial \D _\pm}$, where $\chi _\Om$ denotes the characteristic function of a Borel set, $\Om$, is not log-integrable (with respect to $m$). On the other hand, $m$ is non-extreme. By the previous proposition, $\mr{int} (m_+ , m_-) \neq \{ 0 \}$ is non-trivial, and yet $m_+ \perp m_-$. If $\mr{int} (m_+ , m_-)$ contained a non-trivial $V_+ := V^{m_+}$ or $V_-:= V^{m_-}-$reducing subspace, $\scr{M}$, then the closure, $\scr{M} ^+$ or $\scr{M} ^-$ in the norms of $\scr{H} ^+ (m _{\pm })$ would be a closed $V_+$ or $V_{-}-$reducing subspace. In the first case, Corollary \ref{reducemeas} would then imply that $\scr{M} ^+ = \scr{H} ^+ (\ga)$ for some $0 \leq \ga \leq m_+$. On the other hand, $\mr{int} (\ga , m_-) \supseteq \scr{M}$ is dense in $\scr{M} ^+ = \scr{H} ^+ (\ga )$ so that $\ga \ll _{RK} m_-$. Since RK-absolute continuity is equivalent to absolute continuity by Theorem \ref{RKac}, this contradicts the mutual singularity of $m_+$ and $m_-$. A symmetric argument shows that $\mr{int} (m_+, m_-)$ cannot contain a non-trivial $V_--$reducing subspace either.

Similarly, $m = m_+ + m_-$ can be viewed as the Lebesgue decomposition of $m$ with respect to $m_+$. In this case, $\mr{int} (m , m_+ ) = \scr{H} ^+ (m _+) \neq \{ 0 \}$ since $m _+ \leq m$. However, $\mr{int} (m , m_+)$ cannot be $S= V_m-$reducing as then its closure, $\mr{int} (m , m_+)^{-m}$ would be a closed, $S-$reducing subspace of $H^2 = \scr{H} ^+ (m)$ and the shift has no non-trivial reducing subspaces. (Hence this intersection space cannot contain any non-trivial $S-$reducing subspace.) In fact, $\mr{int} (m_+ , m_-)$ cannot (contractively) contain the space of $\ga-$Cauchy transforms of any non-zero positive measure, $\ga$, as then $\ga \ll _{RK} m_+$ and $\ga \ll _{RK} m_-$, so that $\ga \ll m_+ , m_-$ by Theorem \ref{RKac} and $\ga \equiv 0$ since $m_+$ and $m_-$ are mutually singular. Finally, we cannot have $\mr{int} (m , m_+)$ dense in $H^2$ either as this would imply that $m \ll _{RK} m_+$ which would imply that $m \ll m_+$ by Theorem \ref{RKac}. \\

We can calculate some vectors in $\mr{int} (m_+ , m_-)$ more explicitly. By the proof of Proposition \ref{notred}, we have that $V_+ ^* \oplus V_- ^* U _\vee 1 \in \nbran U _\wedge$, and since $\nbran U_\wedge$ is always $V_+ \oplus V_-$ co-invariant, 
$$ \mr{int} (m_+  , m_- ) \supseteq \bigvee _{j=1} ^\infty  V_+ ^{*j} k_0 ^+ = \bigvee V_- ^{*j} k_0 ^-. $$ Here, $1 = k^m _0$, where $m=m_+ + m_-$, so that $U_\vee 1 = k_0 ^+ \oplus k_0 ^-$. Since the unitaries $V_\pm ^*$ both act as backward shifts on power series, we can compute these elements of the intersection space explicitly. First, the kernel vectors of $\scr{H} ^+ (m _\pm )$ at $0$ are:
\ba k_0 ^+ (z) & = & \frac{1}{2\pi} \int _0 ^\pi \frac{1}{1-z e^{-i \theta}} d\theta \\
& = & \left. \frac{1}{2\pi i} \log (e ^{i\theta} -z ) \right| _{\theta =0} ^{\theta = \pi} \\
& = & \frac{1}{2\pi i} \log \left( \frac{z+1}{z-1} \right), \ea 
where $\log$ is the branch of the logarithm fixed by the choice of the argument function taking values in $[0, 2\pi)$. Here, the branch cut is along the positive real axis, and
$$ \nbre \frac{z+1}{z-1} = \frac{|z| ^2 -1}{|z-1| ^2} <0, $$ is strictly negative for any $z \in \D$ so that this formula defines a holomorphic function in $\D$. (We know, of course, that $k_0 ^+$ must be holomorphic in $\D$.) Since 
$$ 1 = k_0 (z) = k(z, 0) = k^+ (z , 0) + k^- (z,0) = k^+ _0 (z) + k^- _0 (z), $$ it follows that 
\be k_0 ^- (z) = 1 - k_0 ^+ (z) = 1 - \frac{1}{2\pi i} \log \left( \frac{z+1}{z-1} \right). \ee Also note that 
$$ \frac{1}{2} = \frac{1}{i2\pi} \log (-1), $$ so that
$\frac{1}{2} = k_0 ^+ (0) = k_0 ^- (0)$.

Since $V_\pm ^*$ act as backward shifts on power series, it follows that $V_+ ^* k_0 ^+ = -V_- ^* k_0 ^-$, so that 
$$ (V_+ \oplus V_-) ^* k_0 ^+ \oplus k_0 ^- \in \nbran U_\wedge = \bigvee _{h
\in \mr{int} ( m_+ , m_- )} h \oplus -h, $$ as required. 
\end{eg}

\subsection{Lebesgue decomposition of measures and their forms} \label{ss:measforms}

As described in Remark \ref{measform} and Subsection \ref{ss:forms}, if $\mu , \la \geq 0$ are positive, finite and regular Borel measures on the unit circle, $\partial \D$, then one can construct the Lebesgue decomposition of $\mu$ with respect to $\la$ by considering the densely--defined positive quadratic form, $\fq_\mu : \scr{C} (\partial \D ) \times \scr{C} (\partial \D) \rightarrow 0$, with dense form domain $\scr{C} (\partial \D ) \subseteq L^2 (\la )$, the continuous functions on the unit circle. Namely, applying the Simon--Lebesgue decomposition to $\fq_\mu$, viewed as a positive, densely--defined form in $L^2 (\la)$, one obtains,
$$ \fq_\mu = \fq_{\mu ; ac} + \fq _{\mu ;s}, $$ where $\fq_{\mu ;ac}$ is an absolutely continuous (closeable) form and $\fq_s$ is a singular form and moreover, 
$\fq_{\mu ; ac} = \fq_{\mu _{ac}}, \fq_s = \fq_{\mu _s}$, where $\mu = \mu _{ac} + \mu _s$ is the Lebesgue decomposition. 

However, in this paper, since we wish to apply analytic and function theoretic methods, we instead consider the positive quadratic $Z^\la-$Toeplitz form, $\fq_\mu$, associated to $\mu \geq 0$, with dense form domain $\nbdom \fq _\mu = \C [\zeta ]$ or $\nbdom \fq_\mu = A (\D )$, in $H^2 (\la ) \subseteq L^2 (\la )$. As we will show, if $\fq _\mu = \fq_{ac} + \fq_{s}$ is the Simon--Lebesgue form decomposition of $\fq_\mu$ in $H^2 (\la)$, then one can define reproducing kernel Hilbert spaces of $\fq_{ac}$ and $\fq_s-$Cauchy transforms, $\scr{H} ^+ (\fq_{ac})$ and $\scr{H} ^+ (\fq_s)$. The goal of this subsection is to compare the Lebesgue decomposition of $\mu$ with respect to $\la$ with the Simon--Lebesgue decomposition of $\fq_\mu$ in $H^2 (\la )$. 

Let $\mu, \la \geq 0$ be finite and regular Borel measures on $\partial \D$. Consider the positive quadratic form, $\fq_\mu$, with dense form domain, $A (\D ) \subseteq H^2 (\la )$. Observe that $\hat{\cH} (\fq_\mu ) = H^2 (\mu + \la)$ so that $\C [\zeta ]$ and $\scr{K} _\D$ are both dense sets in this space. Consider the Simon--Lebesgue decomposition, $\fq_\mu = \fq_{ac} + \fq_s$, of $\fq_\mu$ in $H^2 (\la )$. By Theorem \ref{formLD}, $\fq_{ac} \geq 0$, is the largest closeable quadratic form bounded above by $\fq_\mu$. Since $\fq_{ac} \leq \fq_\mu$, this implies that $\nbdom \fq_\mu = A (\D ) \subseteq \nbdom \fq_{ac}$, and if $\fq_D = \ov{\fq_{ac}}$ denotes the closure of $\fq_{ac}$, then $A (\D )$ must be a form-core for the closed form $\fq_D$ by the maximality statement in Theorem \ref{formLD}. We define $H^2 (\fq_{ac} ), H^2 ( \fq_s )$ as the Hilbert space completion of the disk algebra, $A (\D)$, modulo vectors of zero length, with respect to the pre-inner products, $\fq_{ac}, \fq_s$, respectively. Since $0 \leq \fq_{ac}, \fq_s \leq \fq_\mu$, we can define the contractive co-embeddings $E_{ac} : H^2 (\mu ) \hookrightarrow H^2 ( \fq_{ac}) $ and $E_s : H^2 (\mu ) \hookrightarrow H^2 (\fq_s)$ by $E_{ac} a = a \in H^2 (\fq_{ac} )$ and $E_s a = a \in H^2 (\mu _s)$. (Here, an element $a \in A (\D )$ could be equal to $0$ as an element of $H^2 (\mu)$, or as an element of the spaces $H^2 (\fq_{ac}), H^2 (\fq_s )$. However, the inequality $0 \leq \fq_{ac} ,\fq_s \leq \fq_\mu$, ensures that if $a \in A (\D)$ is zero as an element of $H^2 (\mu)$, \emph{i.e.} it vanishes $\mu-a.e.$, then $a =0$ as element of both $H^2 (\fq_{ac} )$ and $H^2 (\fq_s)$. A more precise notation would be to let $N_{ac}$ denote the subspace of all elements of $A (\D)$ of zero-length with respect to the $\fq_{ac}-$pre-inner product so that equivalence classes of the form $a + N_{ac}$, $a \in A (\D )$, are dense in $H^2 (\fq_{ac} )$. )

Observe that if $\scr{D} \subseteq A (\D )$ is any supremum-norm dense set, such as $\scr{K} _\D = \bigvee k_z$ or $\C [\zeta ]$, then $\scr{D}$ is dense in $H^2 ( \mu)$, and since the co-embedding $E_{ac} : H^2 (\mu ) \hookrightarrow H^2 (\fq_{ac})$ is a contraction with dense range, $\scr{D}$ will be dense in $H^2 (\fq_{ac})$ and it will be similarly dense in $H^2 (\fq_s)$.

\begin{lemma}
If $\fq_D = \ov{\fq_{ac}}$ is the closure of $\fq_{ac}$, and $\scr{D} \subseteq A (\D )$ is supremum-norm dense, then $\scr{D}$ is a core for $\sqrt{D}$.
\end{lemma}
\begin{proof}
Since $\nbdom \fq_\mu = A (\D )$, and $\fq_{ac} \leq \fq_\mu$ is the largest closeable and positive semi-definite quadratic form, $A (\D)$ is a form-core for $\fq_D$, and hence a core for $\sqrt{D}$. Hence, $A (\D )$ is dense in $\hat{\cH} (\fq_D) = \hat{\cH} (\fq_{ac} )$. Given any $a \in A (\D)$, let $x_n \in \scr{D}$ be a sequence which converges to $a$ in supremum-norm. Then 
\ba 0 & \leq & \| x_n - a \| ^2 _{H^2 (\la )} + \fq_{ac} (x_n -a , x_n -a ) \\
& \leq & \| x_n -a \| ^2 _{H^2 (\la)} + \| x_n - a \| ^2 _{H^2 (\mu )}  \\
& \leq & \| x_n -a \| ^2 _{\infty} (\mu (\partial \D) + \la (\partial \D ) ) \rightarrow 0. \ea 
This proves that $\scr{D}$ is dense in the dense subspace $A (\D ) \subseteq \hat{\cH} (\fq_D)$, and hence $\scr{D}$ is a form-core for $q_D$ and a core for $\sqrt{D}$.
\end{proof}

Given any $h \in H^2 (\fq_{ac} )$ or in $H^2 (\fq_s)$, we can now define the $\fq_{ac}$ or $\fq_s-$Cauchy transform of $h$ as before:
$$ (\scr{C} _{ac} h ) (z) := \fq_{ac} ( k_z , h ), $$ and similarly for $\fq_s$. As in Lemma \ref{lemholo} and Lemma \ref{lemCT}, Cauchy transforms of elements of $H^2 (\fq_{ac}) , H^2 (\fq_s)$ are holomorphic in the unit disk, and if we equip the vector space of $\fq_{ac}-$Cauchy transforms with the inner product 
$$ \ip{ \scr{C} _{ac} x}{\scr{C} _{ac} y}_{ac} := \fq_{ac} (x,y), $$ we obtain a reproducing kernel Hilbert space of analytic functions in the disk, $\scr{H} ^+ (\fq _{ac} )$ with reproducing kernel:
$$ k^{(ac)} (z,w) := \fq_{ac} (k_z , k_w). $$
Finally, since $\fq_\mu = \fq_{ac} + \fq_s$, $\fq_\mu \geq \fq_{ac}, \fq_s \geq 0$, we obtain the following.
\begin{prop} \label{decompid}
The RKHS of $\fq_{ac}$ and $\fq_s-$Cauchy transforms are contractively contained in $\scr{H} ^+ (\mu) = \scr{H} ^+ (q_\mu )$ and $k^\mu = k^{(ac)} + k ^s$ so that 
$$ \scr{H} ^+ (\mu ) = \scr{H} ^+ (\fq_{ac} ) + \scr{H} ^+ (\fq_s). $$
Moreover, if $\mr{e} _{ac} : \scr{H} ^+ (\fq _{ac} ) \hookrightarrow \scr{H} ^+ (\mu)$ and $\mr{e} _s$ are the contractive embeddings, then 
$$ I_{\scr{H} ^+ (\mu )} = \mr{e} _{ac} \mr{e} _{ac} ^* + \mr{e} _s \mr{e} _s ^*.$$ 
\end{prop}
\begin{proof}
To check the decomposition of the identity, it suffices to calculate
\ba k ^\mu (z,w) & = & k^{(ac)} (z,w) + k^s (z,w) \\
& = & \ip{k_z ^{(ac)}}{k_w ^{(ac)}}_{ac} + \ip{k_z ^s}{k_w ^s}_s  \\
& = & \ip{\mr{e} _{ac} ^* k_z ^\mu}{\mr{e} _{ac} ^* k ^\mu _w}_{ac} + \ip{\mr{e} _{s} ^* k_z ^\mu}{\mr{e} _{s} ^* k ^\mu _w}_{s} \\
& = & \ip{k_z ^\mu}{(\mr{e}_{ac} \mr{e}_{ac} ^* + \mr{e} _s \mr{e} _s ^* )k_w ^\mu}_\mu. \ea 
\end{proof}

\begin{thm} \label{rkhsform}
Let $\fq_\mu = \fq_{ac} + \fq_s$ be the Simon--Lebesgue decomposition of the form $\fq_\mu$ with dense form domain $A (\D )$ in $H^2 (\la )$. Then, 
$$ \scr{H} ^+ (\fq_{ac} ) = \mr{int} (\mu , \la) ^{ -\mu } = \left( \scr{H} ^+ (\mu ) \cap \scr{H} ^+ (\la ) \right) ^{ -\| \cdot \| _\mu}. $$ If $\mr{e} : \mr{int} (\mu , \la ) \subseteq \scr{H} ^+ (\mu )\hookrightarrow \scr{H} ^+ (\la )$ is the closed embedding and $\fq_D = \ov{\fq_{ac}}$, then 
$$ D = \scr{C} _\la ^* \mr{e} \mr{e} ^* \scr{C} _\la. $$
\end{thm}

\begin{lemma} \label{leqcore}
Let $\fq_1, \fq_2$ be densely--defined, closed and positive semi-definite quadratic forms in a separable, complex Hilbert space, $\cH$. Then $\fq_1 \leq \fq_2$ if and only if $\fq_1 (x,x) \leq \fq_2 (x,x)$ for all $x$ in a form-core for $\fq_2$.
\end{lemma}

\begin{proof}{ (of Theorem \ref{rkhsform})}
First, since $\fq_{ac}$ is closeable, $\ov{\fq}_{ac} = \fq_B$ for some closed, self-adjoint operator $B \geq 0$. By construction, $A (\D ) \subseteq \nbdom \sqrt{B}$, and $\C [ \zeta ]$, $\scr{K} _\D = \bigvee k_z$ and $A (\D )$ are all cores for $\sqrt{B}$. Since $B \geq 0$ is closed, $\nbdom B$ is also a core for $\sqrt{B}$. It follows that we can identify $\nbdom B$ with a dense subspace of $H^2 (\fq _{ac} )$. Namely, if $x \in \nbdom B \subseteq H^2 ( \la )$, we can find $a_n \in A (\D )$ so that $a_n \rightarrow x$ in $H^2 (\la )$ and $\sqrt{B} a_n \rightarrow \sqrt{B} x$. Since $\ov{\fq_{ac}} = \fq _B$, it follows that $(a_n)$ is a Cauchy sequence in $H^2 (\fq _{ac} )$, and we can identify $x \in \nbdom B$ with the limit, $\hat{x}$, of this Cauchy sequence in the Hilbert space $H^2 (\fq _{ac} )$. Finally, since $\nbdom B$ is a core for $\sqrt{B}$, for any $a \in A (\D ) \subseteq \nbdom \sqrt{B}$, we can find $x_n \in \nbdom B$ so that $x_n \rightarrow a$ and $\sqrt{B} x_n \rightarrow \sqrt{B} a$ and it follows that $\hat{x} _n \rightarrow a$ in $H^2 (\fq _{ac} )$, so that $\nbdom B$ can be identified with a dense subspace of $H^2 (\fq _{ac} )$.

Furthermore, we can then define the $\fq_{ac}-$Cauchy transform of any $x \in \nbdom B$, 
\ba (\scr{C} _{ac} x ) (z) & = & \lim _{n \uparrow \infty} \fq_{ac} (k_z , a_n) = \lim \ip{\sqrt{B} k_z}{\sqrt{B} a_n }_{H^2 (\la )} \\
& = & \ip{\sqrt{B} k_z}{\sqrt{B} x}_{H^2 (\la)}  \\
& = & \ip{k_z}{Bx}_{H^2 (\la )} = (\scr{C} _\la Bx) (z). \ea  
This proves that $\scr{C} _{ac} x \in \scr{H} ^+ (\la )$. Since $\scr{C} _{ac} x \in \scr{H} ^+ (\fq_{ac} ) \subseteq \scr{H} ^+ (\mu )$, it follows that 
$\scr{C} _{ac} \nbdom B \subseteq \mr{int} (\mu , \la )$. Moreover, since $\nbdom B$ can be identified with a dense subspace of $H^2 (\fq_{ac} )$, it follows that $\scr{C} _{ac} \nbdom B \subseteq \scr{H} ^+ (\fq_{ac} ) \cap \scr{H} ^+ (\la ) \subseteq \mr{int} (\mu ,\la )$ is dense in $\scr{H} ^+ (\fq_{ac} )$.

Now consider $\fq_D$, where $D = \scr{C} _\la ^* \mr{e} \mr{e} ^* \scr{C} _\la $ and $\mr{e} : \mr{int} (\mu, \la ) \subseteq \scr{H} ^+ (\mu ) \hookrightarrow \scr{H} ^+ (\la )$, as in the theorem statement. By construction, $\scr{K} _\D$ is a core for $\sqrt{D}$, and it is also a core for $B$, so that this set is a form-core for both $\fq_B = \ov{\fq_{ac}}$ and $\fq_D$. It follows that $\fq_D | _{\scr{K} _\D} \leq \fq_\mu | _{\scr{K} _\D}$ is a positive closeable form so that by maximality and Lemma \ref{leqcore}, $\fq_D \leq \fq_B$ in the form-sense. Also, by construction, $\mr{e} ^* k^\la _z = k ^{\mu \cap \la} _z$, where $k^{\mu \cap \la}$ is the reproducing kernel for the closed subspace $\mr{int} (\mu, \la ) ^{-\mu} \subseteq \scr{H} ^+ (\mu )$. Hence, for any finite subset, $\{ z_1, \cdots , z_n \} \subseteq \D$, if we consider any finite linear combination of Szeg\"o kernels, 
$$ h = \sum _{i=1} ^n c_i k _{z_i}, $$ then 
\ba 0 & \leq & \sum \ov{c_i} c_j k^{\mu \cap \la} (z_i , z_j ) \\ 
& = & \sum \ov{c_i} c_j \ip{k ^{\mu \cap \la} _{z_i} }{k^{\mu \cap \la} _{z_j}}_\mu \\
& = & \fq _D (h , h) \leq \fq _B (h, h) \\
& = & \sum \ov{c_i} c_j \fq _B (k_{z_i} , k_{z_j} ) \\
& = & \sum \ov{c_i} c_j k^{ac} (z_i , z_j ). \ea
That is,
$$ 0 \leq  [k ^{\mu \cap \la} (z_i ,z_j) ] _{1\leq i,j \leq n} =  [\fq_D (k_{z_i} ,k_{z_j}) ] \leq [ \fq_B (k_{z_i} , k_{z_j}) ] = [k ^{ac} (z_i,z_j)], $$
so that $k^{\mu \cap \la} \leq k^{ac}$, and by Aronszajn's inclusion theorem, $\mr{int} (\mu , \la ) ^{-\mu }$ is contractively contained in $\scr{H} ^+ (\fq_{ac} )$ which is in turn contractively contained in $\scr{H} ^+ (\mu )$. Hence, if $\mr{e} _1$ is the first embedding into $\scr{H} ^+ (\fq_{ac} )$ and $\mr{e} _2$ is the second embedding into $\scr{H} ^+ (\mu )$, the composite embedding, $\mr{e} = \mr{e} _2 \mr{e} _1 : \mr{int} (\mu ,\la ) ^{-\mu} \hookrightarrow \scr{H} ^+ (\mu )$ is again a contractive embedding and it must be isometric since $\mr{int} (\mu , \la ) ^{-\mu}$ is a closed subspace of $\scr{H} ^+ (\mu )$. It follows that $\mr{e} _1$ must be an isometric embedding. Indeed, if there is a unit vector $x$ so that $\| \mr{e} _1 x  \| <1$ then 
$$ 1= \|x \| = \| \mr{e} x \| \leq \| \mr{e} _2 \| \| \mr{e} _1  x \| < 1. $$ Similarly $\mr{e} _2$ must be isometric on the range of $\mr{e} _1$. On the other hand, since $\mr{int} (\fq_{ac}, \la) := \scr{H} ^+ (\fq_{ac} ) \cap \scr{H} ^+ (\la )$ is dense in $\scr{H} ^+ (\fq_{ac})$ and $\scr{H} ^+ (\fq_{ac})$ is contractively contained in $\scr{H} ^+ (\mu )$, we must have that $\mr{int} (\fq_{ac} , \la ) \subseteq \mr{int} (\mu, \la) \subseteq \nbran \mr{e} _1$. Hence, by the previous argument, since $\mr{int} (\fq _{ac} , \la ) \subseteq \nbran \mr{e} _1$ is dense in $\scr{H} ^+ (\fq _{ac} )$ and $\mr{e} _2$ is isometric on the range of $\mr{e} _1$, $\mr{e} _2 : \scr{H} ^+ (\fq _{ac} ) \hookrightarrow \scr{H} ^+ (\mu )$ is also an isometric inclusion. In conclusion, $\mr{int} (\mu , \la ) ^{ -\mu}$ and $\scr{H} ^+ (\fq_{ac})$ are both closed subspaces of $\scr{H} ^+ (\mu )$, $\mr{int} (\mu , \la ) ^{-\mu }$ is a closed subspace of $\scr{H} ^+ (\fq_{ac} )$ and $\mr{int} (\fq_{ac}, \la ) \subseteq \mr{int} (\mu , \la )$ is dense in $\scr{H} ^+ (\fq_{ac})$ so that $\mr{int} (\mu , \la ) ^{-\mu} = \scr{H} ^+ (\fq_{ac})$. It follows that $\fq_B = \fq_D$ on $\scr{K} _\D$ so that by Lemma \ref{leqcore} and the uniqueness of representation of closed forms, $D=B$.
\end{proof}

\begin{cor} \label{RKformLD}
If $\mu, \la \geq 0$ are finite, positive and regular Borel measures on $\partial \D$ and $\fq_\mu$ is the densely--defined positive quadratic form associated to $\mu$ with form domain $A (\D ) \subseteq H^2 (\la )$, then the space of $\mu-$Cauchy transforms decomposes as the orthogonal direct sum,
$$\scr{H} ^+ (\mu ) = \scr{H} ^+ (\fq_{ac} ) \oplus \scr{H} ^+ (\fq_s ). $$ In particular, $\scr{H} ^+ (\fq_s) \cap \mr{int} (\mu ,\la ) = \{ 0 \}$.
\end{cor}
\begin{proof}
By Proposition \ref{decompid} and Theorem \ref{formLD}, we have that the identity operator on $\scr{H} ^+ (\mu )$ decomposes as 
$$ I_\mu = \mr{e} _{ac} \mr{e} _{ac} ^* + \mr{e} _s \mr{e} _s ^*, $$ and $\scr{H} ^+ (\fq_{ac} ) = \mr{int} (\mu ,\la ) ^{-\mu}$ is a closed subspace of $\scr{H} ^+ (\mu)$ so that the contractive embedding, $\mr{e} _{ac} : \scr{H} ^+ (\fq_{ac} ) \hookrightarrow \scr{H} ^+ (\mu )$ is an isometry. Hence,
$P_{ac} := \mr{e} _{ac} \mr{e}_{ac} ^*$ is an orthogonal projection onto the range of $\mr{e} _{ac}$ and hence $P_s = I - P_{ac} = \mr{e} _s \mr{e} _s ^*$ is the projection onto the orthgonal complement of $\nbran \mr{e} _{ac}$ in $\scr{H} ^+ (\mu )$. It follows that $\mr{e} _s$ is also an isometric embedding and that we can identify $\scr{H} ^+ (\fq_{ac}), \scr{H} ^+ (\fq _s)$ with the ranges of these isometric embeddings so that 
$$ \scr{H} ^+ (\mu ) = \scr{H} ^+ (\fq_{ac} ) \oplus \scr{H} ^+ (\fq _s ). $$
\end{proof}

\begin{cor} \label{LDvsformLD}
Let $\mu , \la$ be positive, finite and regular Borel measures on the unit circle. The Lebesgue decomposition of $\mu$ with respect to $\la$, $\mu = \mu _{ac} + \mu _s$, coincides with the Simon--Lebesgue decomposition of $\fq_\mu$ with form domain $\nbdom \fq_\mu = A (\D )$ in $H^2 (\la)$, $\fq_\mu = \fq_{ac} + \fq_s$, in the sense that $\fq_{ac} = \fq_{\mu _{ac}}$ and $\fq_s = \fq_{\mu _s}$ if and only if $\mr{int} (\mu, \la)$ is $V^\mu-$reducing.
\end{cor}

\begin{remark}
More generally, one can apply the methods of this section to construct a Lebesgue decomposition for pairs of positive kernel functions $k, K$ on the same set, $X$, see Appendix \ref{kernelLD}.
\end{remark}

\begin{eg}[Lebesgue measure on the half-circles] \label{Lebeg2}
As before, let $m_\pm$ denote normalized Lebesgue measure restricted to the upper and lower half-circles. These are mutually singular measures so that $m_+ = m_{+; s}$ is the singular part of $m_+$ with respect to $m_-$, and yet by Example \ref{LebesgueEg}, $\mr{int} (m_+ , m_-) \neq \{ 0 \}$, so that $\fq_+ = \fq_{m_+}$ has a Simon--Lebesgue decomposition $\fq_+ = \fq_{ac} + \fq _s$ in $H^2 (m_-)$, where $\fq_{ac}$ is non-trivial, by Theorem \ref{rkhsform}. Moreover, in this example, $m_-$ is extreme, so that $H^2 (m_-) = L^2 (m_-)$. This means that while the quadratic form, $\fq_\mu$, associated to $\mu$, with dense form domain, $A (\D ) \subseteq L^2 (m_- ) = H^2 (m _-)$ has non-zero absolutely continuous part, if we instead define the form domain of $\fq_\mu$ to be $\nbdom \fq_\mu = \scr{C} (\partial \D )$, then, with this form domain, $\fq_\mu$ has vanishing absolutely continuous part (since the decompositions of $\fq_\mu$ and $\mu$ always coincide in this case, see Remark \ref{measform}). This shows, that in dealing with these unbounded positive quadratic Toeplitz forms, the choice of form domain is crucial!
\end{eg}

\subsection{Lebesgue decomposition for arbitrary measures}

The question remains: If $\mu, \la \geq 0$ are arbitrary, how can we construct the Lebesgue decomposition of $\mu$ with respect of $\la$ using reproducing kernel theory and their spaces of Cauchy transforms? If $\la$ is non-extreme, or more generally if $\mr{int} (\mu , \la)$ is $V_\mu-$reducing, Theorem \ref{LDviaRKne} provides a satisfying answer. However, as Proposition \ref{notred}, Example \ref{LebesgueEg} and Theorem \ref{rkhsform} show, the intersection of the spaces of $\mu$ and $\la$ Cauchy transforms cannot be reducing in general, and that there are examples of pairs of positive measures $\mu, \la$, for which $\mr{int} (\mu ,\la)$ cannot be equal to, or even contain, the space of Cauchy transforms of any non-zero positive measure. 

By Theorem \ref{RKac}, we do know that if $\mu = \mu _{ac} + \mu _s$ is the Legbesgue decomposition of $\mu$ with respect to $\la$, that $\mu _{ac} \ll _{RK} \la$ so that $\mr{int} (\mu _{ac} , \la ) \subseteq \mr{int} (\mu , \la ) \subseteq \mr{int} (\mu , \la ) ^{-\mu } = \scr{H} ^+ (q_{ac} )$. The final result below provides an abstract characterization of the Lebesgue decomposition for arbitrary pairs of positive measures. 

\begin{thm} \label{genRKLD}
If $\mu = \mu _{ac} + \mu _s$ is the Lebesgue decomposition of $\mu$ with respect to $\la$ and $\fq_\mu = \fq_{ac} +\fq_s$ is the Simon--Lebesgue form decomposition of $\fq_\mu$ in $H^2 (\la )$ then $\fq_{\mu _{ac}} \leq \fq_{ac}$. Moreover, $\scr{H} ^+ (\mu _{ac} )$ is the maximal RKHS, $\cH (k)$, in $\D$ with the following property: $\cH (k) \cap \scr{H} ^+ (\la ) \subseteq \mr{int} (\mu ,\la)$ is dense in $\cH (k)$, $\cH (k) \subseteq \scr{H} ^+ (\mu )$ is contractively contained, and if $\mr{e} : \cH (k) \hookrightarrow \scr{H} ^+ (\mu )$ is the contractive embedding, then $\mr{e} \mr{e} ^*$ is $V_\mu-$Toeplitz. Equivalently, $q_{\mu _{ac}}$ is the largest closeable $Z^\la-$Toeplitz form bounded above by $q_\mu$. \\

Moreover, if $\mr{e} _1 := \mr{e} _{\mu _{ac}}$ and $\mr{e} _2 = \mr{e} _{\mu _s}$ then $I_\mu = \mr{e} _1 \mr{e} _1 ^* + \mr{e} _2 \mr{e} _2 ^*$. Hence, 
we can identify $\scr{H} ^+ (\mu _{ac})$ with the operator--range space $\scr{R} (\mr{e} _1)$ and $\scr{H} ^+ (\mu _s)$ with $\scr{R} (\mr{e} _2) = \scr{R} ^c (\mr{e} _1)$, the complementary space of $\scr{H} ^+ (\mu _{ac})$ in the sense of deBranges and Rovnyak and
$$ \scr{H} ^+ (\mu ) = \scr{H} ^+ (\mu _{ac} ) + \scr{H} ^+ (\mu _s). $$
\end{thm}
\begin{proof}
This follows from the definition of $\fq _{ac}$, Theorem \ref{compRKHS} and Theorem \ref{CTsubspace}.
\end{proof}
\begin{remark}
In the case where the complementary space decomposition of $\scr{H} ^+ (\mu ) = \scr{H} ^+ (\mu _{ac} ) + \scr{H} ^+ (\mu _s )$, appearing in the above theorem statement, is not an orthogonal direct sum, this yields a corresponding decomposition of the quadratic form $\fq _\mu$,
\be \fq _\mu = \fq _{\mu _{ac}} + \fq _{\mu _s}, \label{pseudoperp} \ee where $\fq _{\mu _{ac} } < \fq _{ac}$ and $\fq _\mu = \fq _{ac} + \fq _s$ is the Simon--Lebesgue decomposition of $\fq _\mu$. In this case, the decomposition of Equation (\ref{pseudoperp}) is an example of a `psuedo--orthogonal' Lebesgue decomposition of $\fq _\mu$ as recently defined and studied in \cite{HdeSnoo2}.
\end{remark}

The previous theorem is, while interesting, admittedly not very practical for construction of the Lebesgue decomposition of $\mu$ with respect to $\la$. A simpler, albeit somewhat ad hoc, approach using our reproducing kernel methods is simply to `add Lebesgue measure'. Namely, if $\mu _{ac; \la}$ is the absolutely continuous part of $\mu$ with respect to $\la$, then $\mu_{ac ; \la} = \mu _{ac ; \la + m} - \mu_{ac; m}$ and both $\la + m$ and $m$ are non-extreme so that Theorem \ref{LDviaRKne} applies.

\appendix

\section{Lebesgue decomposition of positive kernels} \label{kernelLD}

Let $K$ be a fixed positive kernel function on a set, $X$. Given any other positive kernel, $k$, on $X$, we can associate to it the densely-defined and positive semi-definite quadratic form, $\fq _k : \nbdom \fq _k \times \nbdom \fq _k \rightarrow \C$, with dense form domain $\nbdom \fq _k := \bigvee _{x \in X} K_x$ in $\cH (K)$, 
$$ \fq _k (K_x , K_y ) := k(x,y). $$ 

One can then apply B. Simon's Lebesgue decomposition of positive quadratic forms to $\fq _k$. Such a Lebesgue decomposition of positive kernels was first considered in \cite[Section 7, Theorem 7.2]{HdeSnoo-forms}. The theorem below provides some more details about this decomposition. 

\begin{thm}
Let $k,K$ be positive kernel functions on a set, $X$. If $\fq _k$ is the densely-defined positive quadratic form of $k$ in $\cH (K)$, as defined above, with Simon--Lebesgue form decomposition $\fq _{k} = \fq _{ac} + \fq _s$, then there are positive kernels, $k^{ac}$ and $k^s$ on $X$, so that $\fq _{ac} = \fq _{k^{(ac)}}$, $\fq _s = \fq _{k^s}$, $k = k^{ac} + k^s$, and
$$ \cH (k) = \cH (k ^{ac} ) \oplus \cH (k^s ). $$ Moreover, $\cH (k ^{ac} ) = \mr{int} (k , K) ^{-k} := \left( \cH (k) \cap \cH (K) \right) ^{-\| \cdot \| _k }$, and if $\mr{e} : \mr{int} (k, K) \hookrightarrow \cH (K)$ is the (closed) embedding, then $\ov{\fq _{ac}} = \fq _{\mr{e}\mr{e} ^*}$. 
\end{thm}
In the above, $\mr{int} (k,K) := \cH (k) \cap \cH (K)$.
\begin{proof}
Let $h := \sum _{i=1} ^n c_i K_{x_i}$ be any finite linear combination of the kernels $K_{x_i}$, $\{ x_i \} _{i=1} ^n \subseteq X$. Then, since $\fq _{ac} \leq \fq _k$, we obtain that 
\ba  \sum \ov{c_i} c_j k (x_i , x_j ) & = & \sum \ov{c_i } c_j \fq _k (x_i , x_j ) \\
& = & \fq _k (h , h) \geq \fq _{ac} (h , h)  \geq 0, \ea where 
$$ 0 \leq \fq _{ac} (h , h) = \sum \ov{c_i} c_j \fq _{ac} ( K_{x_i} , K_{x_j} ). $$
It follows that
$$ k^{ac} (x,y) := \fq _{ac} (K_x , K_y), $$ defines a positive kernel function on $X$ so that $0 \leq k^{ac} \leq k$. Similarly, $k^s (x,y) := \fq _s (K_x , K_y)$ defines a positive kernel function on $X$ so that $0 \leq k^s \leq k$, and since $\fq _k = \fq _{ac} +\fq _s$, we obtain that $k^{ac} + k^s = k$. 

By definition, $\fq _{ac}$ is the largest closeable quadratic form bounded above by $\fq _k$. In particular $\ov{\fq _{ac}} =\fq _D$ is the positive form of some densely-defined, self-adjoint and positive semi-definite operator $D$, so that $\scr{K} _x := \bigvee _{x \in X} K_x$ is a core for $\sqrt{D}$. (Here, $\bigvee$ denotes non-closed linear span.) If $\mr{e} : \mr{int} (k, K) \subseteq \mr{int} (k,K) ^{-k} \hookrightarrow \cH (K)$ is the densely-defined and closed embedding, let $A := \mr{e} \mr{e} ^*$. We claim that $A=D$. First, $A \geq 0$ is self-adjoint, hence closed, and since $\mr{e}$ is trivially a multiplier, we obtain that 
$$ \fq _A ( K_x , K_y ) = \ip{\mr{e} ^* K_x}{\mr{e} ^* K_y}_k = \ip{k ^{\cap} _x}{k^\cap _y}_k = k^\cap (x,y), $$ where $k^\cap$ denotes the reproducing kernel of the subspace $\mr{int} (k,K) ^{-k} \subseteq \cH (k)$, the closure of the intersection space, $\mr{int}(k,K)$ in $\cH (k)$. In particular, since $k^\cap _x = P _\cap k_x$, where $P_\cap : \cH (k) \rightarrow \mr{int} (k , K) ^{-k}$ is the orthogonal projection, it follows that $k^\cap \leq k$, and hence that $\fq _A \leq \fq _k$. Since $\fq _A | _{\scr{K} _X}$ is closeable, it follows, by maximality of the Simon--Lebesgue decomposition, that $\fq _A \leq \fq _D$. This inequality implies that $k^\cap \leq k^{ac}$ as positive kernels on $X$. 

Now suppose that $h \in \nbdom D \subseteq \cH (K)$ and choose $h_n \in \scr{K} _X = \bigvee _{x \in X} K_x$ so that $h_n \rightarrow h$ and $\sqrt{D} h_n \rightarrow \sqrt{D} h$. (This can be done since $\nbdom D$ is a core for $\sqrt{D}$.) If $h_n = \sum _{j=1} ^{m_n} c_j (n) K_{x_j (n)}$, a finite linear combination, then note that 
\ba (Dh) (x) & = & \lim _{n \uparrow \infty } \sum c_j (n) \ip{\sqrt{D} K_x}{\sqrt{D} K_{x_j (n) }}_K  \\
& = & \lim \sum c_j (n)  k^{ac} (x,x_j (n) ) = \lim g_n (x), \ea where 
$$ g_n = \sum c_j (n) k^{ac} _{x_j (n) } \in \cH ( k^{ac} ) \subseteq \cH (k). $$ Moreover,
\ba \| g_n \| ^2 _{k^{ac}} & = & \sum _{i,j}  \ov{c_i (n) } c_j (n) k^{ac} (x_i (n) , x_j (n) ) \\
& = & \ip{\sqrt{D} h_n}{\sqrt{D} h_n} _K \rightarrow \| \sqrt{D} h \| _K ^2, \ea so that the sequence $(g_n ) \subseteq \cH (k^{ac})$ is uniformly bounded in norm. Since $g_n (x) \rightarrow (Dh) (x)$ pointwise in $X$, this and uniform boundedness imply that $g_n$ converges weakly to the function $Dh$. Since Hilbert spaces are weakly closed, the function $Dh \in \cH (k^{ac} ) \subseteq \cH (k)$, and also $Dh \in \cH (K)$ so that $Dh \in \mr{int} (k, K) \subseteq \cH (k ^\cap )$. Hence $Dh \in \cH (k^{ac} )$ and 
$$ \| Dh \| ^2 _{k^{ac}} = \| \sqrt{D} h \| ^2 _K. $$  

Let $\mr{j} _1 : \cH (k^\cap ) \hookrightarrow \cH (k^{ac})$ and $\mr{j} _2 : \cH (k^{ac} ) \hookrightarrow \cH (k)$ be the contractive embeddings. Then $\mr{j} := \mr{j} _2 \mr{j} _1 : \cH (k ^\cap ) \hookrightarrow \cH (k)$ is the isometric embedding of the subspace $\cH (k ^\cap ) \subseteq \cH (k)$ into $\cH (k)$. It follows that $\mr{j} _1$ must be isometric and $\mr{j} _2$ must be isometric on the range of $\mr{j} _1$ in $\cH (k ^{ac} )$. 

We claim that $\nbran \mr{j} _1$ is dense in $\cH (k^{ac})$ so that $\mr{j}_2$ and $\mr{j} _2$ are both isometries. Define a linear map, $V : \scr{K} ^{ac} _X := \bigvee _{x\in X} k^{ac} _x \rightarrow \cH (K)$ by  
$$ V k^{ac} _x := \sqrt{D} K_x, $$ and extending linearly. Since 
$$ \ip{k^{ac} _x}{k^{ac} _y}_{k^{ac}} = k^{ac} (x,y) =  \ip{\sqrt{D} K_x}{\sqrt{D} K_y}_K, $$ it follows that $V$ is an isometry and extends by continuity to an isometry from $\cH (k^{ac} )$ onto the closure of $\sqrt{D} \scr{K} _X$ in $\cH (K)$, which we also denote by $V$. Since $\scr{K} _X$ is a core for $\sqrt{D}$, $V$ is onto $\nbran \sqrt{D} ^{- \| \cdot \| _K}$ in $\cH (K)$. If there exists a $g \in \cH (k^{ac} )$ orthogonal to $\nbran D \subseteq \cH (k ^{ac} )$, then choose a sequence $g_n = \sum _{j=1} ^{m_n} c_j (n) k^{ac} _{x_j (n)} \in \scr{K} ^{ac} _X$ so that $g_n \rightarrow g$ and calculate, for any $h \in \nbdom D$, that
\ba 0 & = & \ip{g}{Dh}_{k^{ac}} = \lim _n \ip{g_n}{Dh}_{k^{ac}} \\
& = & \lim _n \sum _j \ov{c_j(n)} \ip{k^{ac} _{x_j (n)}}{Dh}_{k^{ac}} \\
& = & \lim _n \sum _j \ov{c_j(n)} (Dh) (x_j (n)) \\
& = & \lim _n \sum _j \ov{c_j(n)} \ip{K _{x_j (n)}}{Dh}_K \\
& = & \lim _n \ip{ \sum _j c_j(n) \sqrt{D} K _{x_j (n)}}{\sqrt{D}h}_K \\
& = & \lim _n \ip{Vg_n}{\sqrt{D}h}_K \\
& = & \ip{Vg}{\sqrt{D}h}_K. \ea
This proves that $Vg \in \nbran \sqrt{D} ^{-\| \cdot \| _K}$ is orthogonal to $\sqrt{D} \nbdom D$. However, $\nbdom D$ is a core for $\sqrt{D}$, so that $\sqrt{D} \nbdom D$ is dense in $\nbran \sqrt{D}$. This proves that $Vg=0$, and hence $g=0$. In conclusion, $\nbran D \subseteq \nbran \mr{j} _1 \subseteq \cH (k ^{ac} )$ is dense in $\cH (k^{ac})$ so that both $\mr{j} _1$ and $\mr{j} _2$ are isometric embeddings.  That is, $\cH (k^\cap )$ embeds, as a closed, dense subspace of $\cH (k ^{ac})$, which embeds isometrically into $\cH (k)$ and we conclude that $\cH (k ^\cap) = \cH (k^{ac})$ so that $k^\cap = k^{ac}$ and $\fq _D = \fq _A$. By the uniqueness of Kato's Riesz representation of closed, positive semi-definite forms, $D=A$ as closed operators. 

The fact that $k = k^{ac} + k^s$, implies that if $\mr{e} _{ac} : \cH (k^{ac} ) \hookrightarrow \cH (k)$ is the isometric embedding and $\mr{e} _s : \cH (k^s) \hookrightarrow \cH (k)$ is the contractive (and injective) embedding, that $I = \mr{e} _{ac} \mr{e} _{ac} ^* + \mr{e} _s  \mr{e} _s ^*$. Hence $\mr{e} _s \mr{e} _s ^* = I - P_{ac}$, so that $\cH (k^s)$ also embeds isometrically in $\cH (k)$ as the orthogonal complement of $\cH (k^{ac} )$. 
\end{proof}

\bibliographystyle{abbrvnat}

\Addresses

\end{document}